\documentclass[12pt]{article}
\usepackage[centertags]{amsmath}
\usepackage{amsfonts}
\usepackage{amssymb}

\usepackage{amsthm}
\usepackage[top=3cm, bottom=3cm,left=3cm, right=3cm]{geometry}
\usepackage[all]{xy}

\begin{document}

 \newtheorem{thm}{Theorem}[subsection]
 \newtheorem{cor}[thm]{Corollary}
 \newtheorem{lem}[thm]{Lemma}
 \newtheorem{prop}[thm]{Proposition}
 \theoremstyle{definition}
  \newtheorem{defn}[thm]{Definition}
 \theoremstyle{remark}
 \newtheorem{rem}[thm]{Remark}
 \newtheorem{ex}[thm]{Example}
 \theoremstyle{example}

 \numberwithin{equation}{subsection}
 \def\stackunder#1#2{\mathrel{\mathop{#2}\limits_{#1}}}
 \newcommand{\hooklongrightarrow}{\lhook\joinrel\longrightarrow}

 \title{Transition from local to global of dFHE and dWCHP}
 \author{I. POP}
 \date{}
 \maketitle
\begin{abstract}In a previous paper \cite{Pop4} the author studied the directed weak covering homotopy property (dWCHP)and directed weak fibrations in the category dTop of directed spaces in the sense of M. Grandis \cite{Grand1}, \cite{Grand2}, \cite{Grand3}. This type of maps extend to the category dTop the well known Dold's (or weak) fibrations \cite{Dold1}. In this paper the transition from local to global of the dFHE (directed fiber homotopy equivalence) and the dWCHP are studied by proving two Dold type theorems and respectively a tom Dieck-Kamps-Puppe type theorem \cite{Dieck}. Some new notions of directed topology are defined: d-halo, d-SEP, d-numerable covering, d-shrinkable. Some examples and counterexamples are given.


\end{abstract}

\section{Introduction}
Directed Algebraic Topology is a recent subject which arose from the study of some phenomena in the analysis of concurrency, traffic networks, space-time models, etc.(\cite{Faj1}-\cite{Gou1}). It was systematically developed by Marco Grandis (\cite{Grand1},\cite{Grand2},\cite{Grand3}). Directed spaces have privileged directions and directed paths therefore do not need to be reversible. M.Grandis introduced and studied 'non-reversible' homotopical tools corresponding to ordinary homotopies, fundamental group and fundamental $n$-groupoids: directed homotopies, fundamental monoids and fundamental $n$-categories. Also some directed homotopy constructions were considered: pushouts and pullbacks, mapping cones and homotopy fibres, suspensions and loops, cofibre and fibre spaces. As for directed fibrations, M.Grandis \cite{Grand2} refers to these (more precisely to the so-called lower and upper d-fibrations) only in relation with directed h-pullbacks. But in \cite{Pop3},  the author of this paper defined and studied in detail Hurewicz directed fibrations (bilateral d-fibrations in the sense of definitions given by Grandis). In classical algebraic topology, besides the homotopy covering property (CHP)\cite{Hur}, other properties of covering/lifting homotopy have also been studied and proved very interesting :weak CHP (WCHP)\cite{Dold1}, \cite{Dieck}, \cite{Kamp1}, \cite{James}; rather weak CHP (RWCHP)\cite{Brown1},\cite{Brown2},\cite{Pop1},\cite{Kie}; very weak CHP (VWCHP) \cite{Kie}; approximate fibrations (AHLP)\cite{CorDuv},\cite{CorDuv2},\cite{Pop2}. In \cite{Pop4} the author introduced and studied the directed weak covering homotopy property (dWCHP)and directed weak fibrations. In this paper the transition from local to global of the dFHE (directed fiber homotopy equivalence) and the dWCHP are studied by proving two Dold type theorems ( Theorems \ref{thm.3.1.13} and \ref{thm.3.1.15} ) and respectively a tom Dieck-Kamps-Puppe type theorem (Theorem \ref{thm.3.2.4}). These rather seemingly restrictive notions have been introduced by the author in order to adjust some undirected case proofs to directed algebraic topology. But these are closely linked to the respective usual notions with which they coincide if the d-structures contain all the paths of the underlying spaces. Moreover, in all the cases, examples and counterexamples are given, or at least explanatory commentaries are made.

The basics of the Directed Algebraic Topology which we will use are taken from the works of Marco Grandis \cite{Grand1},\cite{Grand2},\cite{Grand3}.

A \textsl{directed space} or a \textsl{d-space}, is a topological space $X$ equipped with a set $dX$ of continuous maps $a:\mathbf{I}=[0,1]\rightarrow X$, called \textsl{directed paths}, or \textsl{d-paths}, satisfying the following three axioms:

(i)(constant paths) every constant map $\mathbf{I}\rightarrow X$ is a directed path,

(ii)(reparametrisation) the set $dX$ is closed under composition with (weakly)increasing maps $\mathbf{I}\rightarrow \mathbf{I}$ ,

(iii) (concatenation) the set $dX$ is closed under concatenation (the product of consecutive paths, which will be denoted by $\ast$).

We use the notations $\underline{X}$ or $\uparrow X$ if $X$ is the underlying topological space; if $\underline{X}$ (or $\uparrow X$) is given, then the set of directed paths is denoted by $d\underline{X}$ (resp.$d(\uparrow X)$)and the underlying space by $|\underline{X}|$ (resp.$|\uparrow X|$).

The standard $d$-interval with the directed paths given by increasing (weakly) maps $\mathbf{I}\rightarrow \mathbf{I}$ is denoted by $\uparrow \mathbf{I}=\uparrow [0,1]$.

A \textsl{directed map}, or $d-map$, $f:\underline{X}\rightarrow \underline{Y}$, is a continuous mapping between d-spaces which preserves the directed paths: if $a\in d\underline{X}$, then $f\circ a\in d\underline{Y}$. The category of directed spaces and directed maps is denoted by $d\mathbf{Top}$ (or $\uparrow \mathbf{Top}$). A directed path $a\in d\underline{X}$ defines a directed map $a:\uparrow \mathbf{I}\rightarrow \underline{X}$ which is also a \textsl{path} of $\underline{X}$.

For two points $x,x'\in \underline{X}$ we write $x\preceq x'$ if there exists a directed path from $x$ to $x'$. The equivalence relation $\simeq $ spanned by $\preceq $ yields the partition of a d-space $\underline{X}$ in its directed path components and a functor $\uparrow \Pi_0:d\mathbf{Top}\rightarrow \mathbf{Set}$, $\uparrow \Pi_0(\underline{X})=|\underline{X}|/\simeq $. A non-empty d-space $\underline{X}$ is a \textsl{directed path connected} if $\uparrow \Pi_0(\underline{X})$ contains only one element.

The \textsl{directed cylinder} of a d-space $\underline{X}$ is the d-space $\uparrow (|\underline{X}|\times \mathbf{I})$ , denoted by $\underline{X}\times \uparrow \mathbf{I}$ or $\uparrow \mathbf{I}\underline{X}$, for which a path $\mathbf{I}\rightarrow |\underline{X}|\times \mathbf{I}$ is directed if and only if its components $\mathbf{I}\rightarrow \underline{X}$, $\mathbf{I}\rightarrow \mathbf{I}$ are directed. The directed maps $\partial^{\alpha}:\underline{X}\rightarrow  \underline{X}\times \uparrow \mathbf{I}$, $\alpha=0,1$, defined by $\partial^{\alpha}(x)=(x,\alpha)$, are called the \textsl{faces} of the cylinder.

If $f,g:\underline{X}\rightarrow \underline{Y}$ are directed maps, a \textsl{directed homotopy} $\varphi$  from $f$ to $g$, denoted by $\varphi:f\rightarrow g$, or $\varphi:f\preceq g$, is a d-map $\varphi:\underline{X}\times \uparrow \mathbf{I}\rightarrow \underline{Y}$ such that $\partial^0\circ \varphi=f$ and $\partial^1\circ \varphi=g$. The equivalence relation defined by the d-homotopy preorder $\preceq $ is denoted by $f\simeq_d g$ or simply $f\simeq g$ . This means that there exists a finite sequence $f\preceq f_1\succeq f_2\preceq f_3\succeq...g$.

Two d-spaces $\underline{X}$ and $\underline{Y}$ are \textsl{d-homotopy equivalent} if there exist d-maps $f:\underline{X}\rightarrow \underline{Y}$ and $g:\underline{Y}\rightarrow \underline{X}$ such that $g\circ f\simeq_d 1_{\underline{X}}$ and $f\circ g\simeq_d 1_{\underline{Y}}$.

\textsc{Standard models}. The spaces $\textbf{R}^n$, $\textbf{I}^n$, $\textbf{S}^n$  have their \textsl{natural} d-structure, admitting all (continuous) paths. $\textbf{I}$ is called the \textsl{natural} interval. The \textsl{directed real line}, or \textsl{d-line} $\uparrow \mathbf{R}$ is the Euclidean line with directed paths given by the increasing maps $\mathbf{I}\rightarrow \mathbf{R}$ (with respect to natural orders). Its cartesian power in d\textbf{Top}, the \textsl{n-dimensional real d-space }$\uparrow \mathbf{R}^n$ is similarly described (with respect to the product order, $x\leq x'$ iff $x_i\leq x'_i$ for all $i$). The \textsl{standard d-interval} $\uparrow \mathbf{I}=\uparrow [0,1]$ has the subspace structure of the d-line; the \textsl{standard d-cube } $\uparrow \mathbf{I}^n $ is its n-th power, and a subspace of $\uparrow \mathbf{R}^n$. The \textsl{standard directed circle} $\uparrow \textbf{S}^1$ is the standard circle with the \textsl{anticlockwise structure}, where the directed paths $a:\mathbf{I}\rightarrow \mathbf{S}^1$ move this way, in the plane: $a(t)=[1,\vartheta (t)]$, with an increasing function $\vartheta$ (in polar coordinates).

A \textsl{directed quotient} $\underline{X}/R$ has the quotient structure, formed of finite concatenations of projected d-paths; in particular, for a subset $A\subset |\underline{X}|$, by $\underline{X}/A$ is denoted the quotient of $\underline{X}$ which identifies all points of $A$. In particular, $\uparrow \mathbf{S}^n=(\uparrow \mathbf{I}^n/\partial \mathbf{I}^n), (n>0)$. The standard circle has another d-structure of interest, induced by $\mathbf{R}\times \uparrow \mathbf{R}$ and called the ordered circle $\uparrow \mathbf{O}^1\subset \mathbf{R}\times \uparrow \mathbf{R}$. It is the quotient of $\uparrow \mathbf{I}+\uparrow \mathbf{I}$ which identifies lower and upper endpoints, separately.

The forgetful functor $U:d\mathbf{Top}\rightarrow \mathbf{Top}$ has adjoints: on the  left  $c_0(X)$ with \textsl{d-discrete structure} of constant paths,  and on the right $C^0(X)$ with the \textsl{natural d-structure} of all paths.

Reversing d-paths, by involution $r:\mathbf{I}\rightarrow \mathbf{I}, r(t)=1-t$, gives the \textsl{reflected}, or opposite, d-space; this forms a (covariant) involutive endofunctor, called \textsl{reflection} $R:d\mathbf{Top}\rightarrow d\mathbf{Top}, R(\underline{X})=\underline{X}^{op},(a\in d(\underline{X }^{op})\Leftrightarrow a^{op}:=a\circ r\in d\underline{X}$). A d-space is \textsl{symmetric} if it is invariant under reflection. It is \textsl{reflexive}, or \textsl{self-dual}, if it is isomorphic to its reflection, which is more general. The d-spaces $\uparrow \mathbf{R}^n$,$\uparrow \mathbf{I}^n$, $\uparrow\mathbf{S}^n$ and $\uparrow \mathbf{O}^1$ are all reflexive.

\section{Directed fibrations and Directed weak fibrations  }

In this section we resume the papers  \cite{Pop3} and \cite{Pop4} including definitions and results that are necessary or interesting for the present paper.

\subsection{Directed fibrations \cite{Pop3}}

\begin{defn}\label{def.2.1.1}Let $p:\underline{E}\rightarrow \underline{B}$,$f:\underline{X}\rightarrow \underline{B}$ be directed maps. A d-map $f':\underline{X}\rightarrow \underline{E}$ is called a directed lift of $f$ with respect to $p$ if $p\circ f'=f$.
\end{defn}
\begin{defn}\label{def.2.1.2}A directed map $p:\underline{E}\rightarrow \underline{B}$ is said to have the directed homotopy lifting property with respect to a d-space $\underline{X}$ if, given d-maps $f':\underline{X}\rightarrow \underline{E}$ and $\varphi:\underline{X}\times \uparrow \mathbf{I}\rightarrow \underline{B}$, and $\alpha\in \{0,1\}$, such that $\varphi \circ \partial^{\alpha}=p\circ f'$, there is a directed lift of $\varphi$, $\varphi':\underline{X}\times \uparrow \mathbf{I}\rightarrow \underline{E}$, with respect to $p$, $p\circ \varphi'=\varphi$, such that $\varphi'\circ\partial^\alpha=f'$.
$$
\xymatrix{
\underline{E} \ar[rr]^p & & \underline{B}\\
\underline{X} \ar[u]^{f'}_{\ \  =} \ar[rr]_{\partial^{\alpha}} & &
\underline{X}\times \uparrow
\textbf{I}\ar[ull]^{\varphi'}\ar[u]_{\varphi}^{=\,\,\,\,\,\,\,} }
$$
A directed map $p:\underline{E}\rightarrow \underline{B}$ is called a directed (Hurewicz) fibration if $p$ has the directed homotopy lifting property with respect to all directed spaces (dHLP). This property is also called the directed covering homotopy property (dCHP).
\end{defn}
\textsc{Properties of directed fibrations}.

P1. If $p:\underline{E}\rightarrow \underline{B}$ has the directed homotopy lifting property with respect to $\underline{X}$ and $f_0,f_1:\underline{X}\rightarrow \underline{B}$ are directed homotopic, $f_0\simeq_d f_1$, then $f_0$ has a directed lift with respect to $p$ iff $f_1$ has this property.

P2. Let $p:\underline{E}\rightarrow \underline{B}$ be a directed fibration and $a\in d\underline{B}$ with $a(\alpha)=p(e_\alpha),e_\alpha\in |\underline{E}|$, and $\alpha\in \{0,1\}$. Then there exists a directed path $a_\alpha\in d\underline{E}$ which is a lift of $a$, $p\circ a_\alpha=a$, with the $\alpha$-endpoint $e_\alpha$, $a_\alpha(\alpha)=e_\alpha$.

\textsc{Examples of directed fibrations}.

E1. Let $\underline{F}$ and $\underline{B}$  be arbitrary directed spaces and let $p:\underline{B}\times \underline{F}\rightarrow \underline{B}$ be the projection. Then $p$ is a directed fibration.

E2. Let $p:E\rightarrow |\underline{B}|$  be a Hurewicz fibration. For the space $E$ consider the maximal d-structure compatible with $d\underline{B}$ and $p$, i.e., $d(\uparrow E)=\{a\in E^I/p\circ a\in d\underline{B}\}$. Then $p:\uparrow E\rightarrow\underline{B}$ is a directed fibration.

E3. If $p:\underline{E}\rightarrow \underline{B}$ is a directed fibration, then the opposite map $p:\underline{E}^{op}\rightarrow \underline{B}^{op}$ is also a directed fibration.

For an intrinsic characterization of the dHLP we need some notations.

Given a d-map $p:\underline{E}\rightarrow \underline{B}$ and $\alpha\in \{0,1\}$, we consider the following d-subspace of the product (in d\textbf{Top}) $\underline{E}\times \underline{B}^{\uparrow \mathbf{I}}$
$$\underline{B}_\alpha=\{(e,a)\in \underline{E}\times \underline{B}^{\uparrow \mathbf{I}}|a(\alpha)=p(e)\} $$
(The d-structure of $\underline{B}^{\uparrow \mathbf{I}}$ is given by the exponential law, $d\mathbf{Top}(\uparrow \mathbf{I},d\mathbf{Top}(\uparrow \mathbf{I},\underline{B}))\approx d\mathbf{Top}(\uparrow \mathbf{I}\times \uparrow \mathbf{I},\underline{B})$,\cite{Grand1}).

\begin{defn}\label{def.2.1.3}
A directed lifting pair for a directed map $p:\underline{E}\rightarrow\underline{B}$ is a pair of d-maps

\begin{equation}
\lambda_\alpha:\underline{B}_\alpha\rightarrow \underline{E}^{\uparrow \mathbf{I}}, \alpha=0,1,
\end{equation}
satisfying the following conditions:

\begin{equation}
\lambda_\alpha(e,\omega)(\alpha)=e,
\end{equation}
\begin{equation}
p\circ \lambda_\alpha(e,\omega)=\omega,
\end{equation}
for each $(e,\omega)\in \underline{B}_\alpha$.
\end{defn}
\begin{thm}\label{thm.2.1.4}(i) A directed map $p:\underline{E}\rightarrow \underline{B}$ is a directed fibration if and only if there exists a directed lifting pair for $p$.

(ii) If $p:\underline{E}\rightarrow \underline{B}$ is a directed fibration, the d-spaces $\underline{B}_0$ and $\underline{B}_1$ are d-homotopiy equivalent.
\end{thm}

\subsection{Directed weak fibrations \cite{Pop4}}

\begin{defn}\label{def.2.2.1}
Let $p:\underline{E}\rightarrow \underline{B}$ and $f,g:\underline{X} \rightarrow \underline{E} $ directed maps such that $p\circ f=p\circ g$. If we suppose that $f\simeq_d g$ by a sequence of directed homotopies  $f\preceq f_1\succeq f_2\preceq f_3\succeq...g$, $p\circ f_k=p\circ f=p\circ g$,  with all the homotopies $\varphi_1:f\preceq f_1,\varphi_2:f_2\preceq f_1...,\varphi_k:f_k\preceq f_{k\pm 1},....$ satisfying  the conditions $(p\circ\varphi_k)(x,t)=p(f(x))$ , $k=1,2,...$, $(\forall) x\in \underline{X},(\forall) t\in [0,1]$, then we denote this by $f\stackunder{(p)}{\simeq_d}g $ and say that $f$ is vertically directed homotopic to $g$.
\end{defn}
\begin{defn}\label{def.2.2.2}
A directed map $p:\underline{E}\rightarrow \underline{B}$ is said to have the directed weak covering homotopy property with respect to a d-space $\underline{X}$ if, given d-maps $ f':\underline{X}\rightarrow \underline{E}$ and $\varphi:\underline{X}\times \uparrow \mathbf{I}\rightarrow \underline{B}$, and $\alpha \in \{0,1\}$, such that $\varphi\circ \partial^\alpha=p\circ f'$, there is a directed lift of $\varphi$, $\varphi': \underline{X}\times \uparrow \mathbf{I}\rightarrow \underline{E}$, with respect to $p$, $p\circ \varphi'=\varphi$, such that $\varphi'\circ \partial^\alpha$ and $f'$ are vertically  directed homotopic $ \varphi'\circ \partial^\alpha \stackunder{(p)}{\simeq_d} f'$.
$$
\xymatrix{
\underline{E} \ar[rr]^p & & \underline{B}\\
\underline{X} \ar[u]^{f'}_{\ \  \stackunder{(p)}{\simeq_d}} \ar[rr]_{\partial^{\alpha}} & &
\underline{X}\times \uparrow
\textbf{I}\ar@{-->}[ull]^{\varphi'}\ar[u]_{\varphi}^{=\,\,\,\,\,\,\,} }
$$
\end{defn}
\begin{prop}\label{prop.2.2.3} If $p:\underline{E}\rightarrow \underline{B}$ has the directed weak covering homotopy property with respect to $\underline{X}$ and $f_0,f_1:\underline{X}\rightarrow \underline{B}$ are directed homotopic, $f_0\simeq_d f_1$, then $f_0$ has a direct lift with respect to $p$ if and only if $f_1$ has this property.
\end{prop}

\begin{defn}\label{def.2.2.4} A directed map $p:\underline{E}\rightarrow \underline{B}$ is called a directed weak fibration or a directed Dold fibration if $p$ has the directed weak covering homotopy property with respect to every directed space (dWCHP).
\end{defn}
\begin{cor}\label{cor.2.2.5}
Let $p:\underline{E} \rightarrow \underline{B}$ be a directed weak fibration and  $a\in d \underline{B}$ a directed path with $a(\alpha)=p(e_\alpha)$ for a point $e_\alpha \in |\underline{E}|$ and $\alpha\in \{0,1\}$. Then $a$ admits a directed lift $a'_\alpha\in d\underline{E}$, $p\circ a'_\alpha=a$, whose  $\alpha-end$ point, $a'_\alpha(\alpha)$, is in the same directed path component of $\underline{E}$ as $e_\alpha$.
\end{cor}
\begin{ex}\label{ex.2.2.6}Let $p:E\rightarrow B$ be a weak fibration and $\uparrow B$ a d-structure on the space $B$. For the space $E$ consider the maximal d-structure compatible with $p$, i.e., $a\in \uparrow E $ iff $p\circ a\in d(\uparrow B)$. Then the directed map $p:\uparrow E\rightarrow \uparrow B$ is a directed weak fibration.
\end{ex}
\begin{ex}\label{ex.2.2.7}Consider in the directed space $\uparrow \mathbf{R}\times \mathbf{R}$ the subspaces $\underline{B}=\uparrow \mathbf{R}$ and $\underline{E}=\{(x,y)|x.y\geq 0\}$ , and the directed map $p:\underline{E}\rightarrow \underline{B}$, $p(x,y)=x$. This is not a directed fibration, but it is a directed weak fibration.
\end{ex}
\begin{defn}\label{def.2.2.8}
Let $p:\underline{E}\rightarrow \underline{B}$, $p':\underline{E}'\rightarrow \underline{B}$ directed maps. We say that $p$ is directed dominated by $p'$ (or $p'$ dominates $p$) if there exist directed maps $f:\underline{E}\rightarrow \underline{E}',g:\underline{E}'\rightarrow \underline{E}$ over $\underline{B}$, $p'\circ f=p,p\circ g=p'$ such that $g\circ f\stackunder{(p)}{\simeq_d}id_{\underline{E}}$.
\end{defn}
\begin{thm}\label{thm.2.2.8}If $p:\underline{E}\rightarrow \underline{B}$ is directed dominated by $p':\underline{E}'\rightarrow \underline{B}$ and if $p'$ has the dWCHP with respect to $\underline{X}$, then $p$ has the same property. Consequently if $p'$ is a directed weak fibration then $p$ is also a directed weak fibration.
\end{thm}

\begin{cor}\label{cor.2.2.10}If $p:\underline{E}\rightarrow \underline{B}$ is directed dominated by a directed fibration $p':\underline{E}'\rightarrow \underline{B}$, then $p$ is a directed weak fibration.
\end{cor}
\begin{prop}\label{prop.2.2.11}Let $p:\underline{E}\rightarrow \underline{B}$ be a directed map with the dWCHP with respect to $\underline{X}$. Let $g':\underline{X}\rightarrow \underline{E}, \psi:\underline{X}\times (\uparrow \mathbf{I})^{op}\rightarrow \underline{B}$ directed maps, and $\alpha\in \{0,1\}$, such that $\psi\circ\partial^\alpha =p\circ g'$,where $\partial^\alpha$ denotes the faces of the inverse cylinder $\underline{X}\times (\uparrow \mathbf{I})^{op}$. Then there exists a directed map $\psi':\underline{X}\times (\uparrow \mathbf{I})^{op}\rightarrow \underline{E}$ satisfying the conditions $p\circ \psi'=\psi$ and $\psi'\circ \partial^\alpha \stackunder{(p)}{\simeq_d} g'$.
\end{prop}

\begin{cor}\label{cor.2.2.12}The reflection endofunctor $R:d\mathbf{Top}\rightarrow d\mathbf{Top}$ conserves the property of directed weak fibration.
\end{cor}

\begin{prop}\label{prop.2.2.13}Let $p:\underline{E}\rightarrow \underline{B}$ be a directed weak fibration and $f:\underline{B}'\rightarrow \underline{B}$ a directed map. Denote the pullback $\underline{E}\prod_{\underline{B}}\underline{B}'$ by $\underline{E}_f$, i.e., $\underline{E}_f=\{(e,b')\in \underline{E}\times \underline{B}'|p(e)=f(b')\}$ with the d-structure as subspace of the product $\underline{E}\times \underline{B}'$. Then the projection $p_f:\underline{E}_f\rightarrow \underline{B}'$ is a directed weak fibration.
\end{prop}
\subsection{A further charcterisation of dWCHP \cite{Pop4}}

\begin{defn}\label{def.2.3.1}A directed homotopy $\varphi:\underline{X}\times \uparrow \mathbf{I}\rightarrow \underline{Y}$ is called semistationary if either  $\varphi(x,t)=\varphi(x,\frac{1}{2}),(\forall)x\in X,t\in [0,\frac{1}{2}]$ or $\varphi(x,t)=\varphi(x,\frac{1}{2}),(\forall) x\in X,t\in [\frac{1}{2},1]$. In the first case we say that $\varphi$ is lower semistationary and in the second case we say that $\varphi$ is upper semistationary.
 \end{defn}
\begin{thm}\label{thm.2.3.2}
A directed map $p:\underline{E}\rightarrow \underline{B}$ has the directed weak covering homotopy property (dWCHP) with respect to a directed space $\underline{X}$ if and only if $p$ has the directed covering homotopy property(dCHP) with respect to $\underline{X}$ for all semistationary directed homotopies.
\end{thm}

\begin{cor}\label{cor.2.3.3}Let $p:E\rightarrow B$ be a directed weak fibration. Let $f':\underline{X}\rightarrow \underline{E}$ a d-map.
Suppose that for $\alpha\in \{0,1\}$ there is a map $\varphi:\underline{X}\times \uparrow \mathbf{I}\rightarrow \underline{B}$ such that
$\varphi\circ \partial^\alpha =p\circ f'$. Also suppose $\varepsilon \in (0,1)$  be chosen.

(i) If $\alpha=0$ and $\varphi $ is stationary on the interval $[0,\varepsilon]$, then there exists $\varphi':\underline{X}\times \uparrow \mathbf{I}
\rightarrow \underline{E}$ satisfying the relations $\varphi'\circ \partial^0=f'$ and $p\circ \varphi'=\varphi$.

(ii) If $\alpha=1$ and $\varphi$ is stationary on the interval $[\varepsilon,1]$, then there exists $\varphi':\underline{X}\times \uparrow \mathbf{I}
\rightarrow \underline{E}$ satisfying the relations $\varphi'\circ \partial^1=f'$ and $p\circ \varphi'=\varphi$.
\end{cor}

\subsection{An intrinsic charcterisation of dWCHP: directed semistationary lifting pair \cite{Pop4}}

Given a d-map $p:\underline{E}\rightarrow \underline{B}$ and $\alpha\in\{0,1\}$, we consider the following d-subspaces of the directed product $\underline{E}\times \underline{B}^{\uparrow \mathbf{I}}$:

$$\underline{B}^s_0=\{(e,\omega)\in \underline{E}\times \underline{B}^{\uparrow \mathbf{I}}|\omega(t)=p(e),(\forall)t\in [0,\frac{1}{2}]\},$$
and
$$ \underline{B}^s_1=\{(e,\omega)\in \underline{E}\times \underline{B}^{\uparrow \mathbf{I}}|\omega(t)=p(e),(\forall)t\in [\frac{1}{2},1]\}.$$
(The d-structure of $\underline{B}^{\uparrow \mathbf{I}}$ is given by the exponential law, $d\textbf{Top}(\uparrow \mathbf{I},d\mathbf{Top}(\uparrow \mathbf{I},\underline{B}))\approx d\mathbf{Top}(\uparrow \mathbf{I}\times \uparrow \mathbf{I},\underline{B})$,\cite{Grand1}).
\begin{defn}\label{def.2.4.1}A directed semistationary lifting pair for a directed map $p:\underline{E} \rightarrow \underline{B}$ consists of a pair of d-maps
\begin{equation}
\lambda^s_\alpha:\underline{B}^s_\alpha\rightarrow \underline{E}^{\uparrow \mathbf{I}},\alpha=0,1,
\end{equation}
satisfying the following conditions:
\begin{equation}
\lambda^s_\alpha(e,\omega)(\alpha)=e,
\end{equation}
\begin{equation}
p\circ \lambda^s_\alpha(e,\omega)=\omega,
\end{equation}
for each $(e,\omega)\in \underline{B}^s_\alpha.$
\end{defn}
\begin{thm}\label{thm.2.4.2}A directed map $p:\underline{E}\rightarrow \underline{B}$ is a directed weak fibration if and only if there exists a directed semistationary lifting pair for $p$.
\end{thm}

\begin{cor}\label{cor.2.4.3}Let $p:\underline{E}\rightarrow \underline{B}$ be a directed weak fibration . For  $\varepsilon\in (0,1)$ consider the
following directed subspaces of $\underline{E}\times \underline{B}^{\uparrow \mathbf{I}}$:

$\underline{B}_{\varepsilon}=\{(e,\omega)\in \underline{E}\times \underline{B}^{\uparrow \mathbf{I}} |\ \omega(t)=p(e),(\forall)t\in [0,\varepsilon]\}$ and

$\underline{B}^{\varepsilon}=\{(e,\omega)\in \underline{E}\times \underline{B}^{\uparrow \mathbf{I}} |\ \omega(t)=p(e),(\forall)t\in [\varepsilon,1]\}$.

Then there exists a pair of directed maps $\lambda_{\varepsilon}:\underline{B}_{\varepsilon}\rightarrow \underline{E}^{\uparrow \mathbf{I}}$ and
$\lambda^{\varepsilon}:\underline{B}^{\varepsilon}\rightarrow \underline{E}^{\uparrow \mathbf{I}}$, satisfying $\lambda_\varepsilon(e,\omega)(0)=e$,
$p\circ \lambda_\varepsilon(e,\omega)=\omega$ and, respectively $\lambda^{\varepsilon}(e',\omega')(1)=e',p\circ \lambda^{\varepsilon}(e',\omega')=\omega'$.
\end{cor}

\begin{thm}\label{thm.2.4.4}
If $p:\underline{E}\rightarrow \underline{B}$ is a directed weak fibration with $\underline{E}\neq \emptyset$ and $\underline{B}$ is a directed path connected space, then $p$ is surjective and fibres of $p$ have all the same directed homotopy type.
\end{thm}

\begin{rem}\label{rem.2.4.5}A theorem similar to Theorem \ref{thm.2.4.2} exists also in the undirected case. But in that case it is sufficient to have a  lifting function for the stationary path on the interval $[0,\frac{1}{2}]$ since there the spaces $B^s_0$ and $B^s_1$ are homeomorphic by the correspondence $(e,\omega)\in B^s_0\rightarrow (e,\omega^{op})\in B^s_1$. And if $\lambda^s_0$ exists, then $\lambda^s_1$ is defined by $(\lambda^s_1(e,\omega)=(\lambda^s_0(e,\omega^{op}))^{op}$. In the general directed case, for an arbitary directed map $p:\underline{E}\rightarrow \underline{B}$, the spaces $\underline{B}^s_0$ and $\underline{B}^s_1$ are independent. But if $p$ is a directed weak fibration, then we have the following theorem.
\end{rem}
\begin{thm}\label{thm.2.4.6}If $p:\underline{E} \rightarrow \underline{B}$ is a directed weak fibration, then the d-spaces $\underline{B}^s_0$ and $\underline{B}^s_1$ are d-homotopy equivalent.
\end{thm}

\begin{cor}\label{cor.2.4.7} If $p:\underline{E}\rightarrow \underline{B}$ is a directed weak fibration and $\varepsilon\in (0,1)$, then the d-spaces
$\underline{B}_{\varepsilon}$ and $\underline{B}^{\varepsilon}$ are d-homotopy equivalent.
\end{cor}

\subsection{Directed fiber homotopy equivalence dFHE \cite{Pop4}}

\begin{defn}\label{def.2.5.1}Let $p:\underline{E}\rightarrow \underline{B}$, $p':\underline{E}'\rightarrow \underline{B}$ and $f:E\rightarrow E'$ directed maps with $p'\circ f=p$ ($f$ is a morphism from $p$ to $p'$). We say that $f$ is a directed fibre homotopy equivalence if there exists $g:\underline{E}'\rightarrow \underline{E}$ a morphism from $p'$ to $p$, such that $g\circ f \stackunder{(p)}{\simeq_d}id_{\underline{E}}$ and $f\circ g\stackunder{(p')}{\simeq_d}id_{\underline{E}'}$.
\end{defn}

\begin{thm}\label{thm.2.5.2}Let $p:\underline{E}\rightarrow \underline{B}$ , $p':\underline{E}'\rightarrow \underline{B}$ be directed weak fibrations.
Then a directed map $f:\underline{E}\rightarrow \underline{E}'$ over $\underline{B}$, $p'\circ f=p$, is a directed fibre homotopy equivalence if and only
if it is an ordinary directed homotopy equivalence.
\end{thm}

\begin{defn}\label{def.2.5.3} A d-map $p:\underline{E}\rightarrow \underline{B}$ is called \emph{d-shrinkable} if one of the following equivalent properties is satisfied:

(a) $p$ is a directed fibre homotopy equivalence (viewed as a d-map over $\underline{B}$ into $id_{\underline{B}}$),

(b) $p$ is directed dominated by $id_{\underline{B}}$ ,

(c) there is a d-section $s:\underline{B}\rightarrow \underline{E}$ , $p\circ s=id_{\underline{B}}$, and a vertical homotopy d-equivalence
$s\circ p\simeq_d id_{\underline{E}}.$

\end{defn}

By Theorem \ref{thm.2.5.2} we obtain:

\begin{cor}\label{cor.2.5.4} If $p:\underline{E}\rightarrow \underline{B}$ is directed weak fibration, then $p$ is d-shrinkable if and only if $p$ is
a directed homotopy equivalence.
\end{cor}

\begin{cor}\label{cor.2.5.5}  Let $p:\underline{E}\rightarrow \underline{B}$, $p':\underline{E}'\rightarrow \underline{B}$ be directed weak fibrations
and $f:\underline{E}\rightarrow \underline{E}'$ a d-map over $B$, $p'\circ f=p$. Suppose that $\underline{B}$ is directed contractible to a point
$b\in \underline{B}$ (i.e., $id_{\underline{B}}\simeq_d c_b$, where $c_b$ is the constant map $c_b(\underline{B})=b$), and that the restriction \qquad
$f_b:\uparrow p^{-1}(b)\rightarrow \uparrow p'^{-1}(b)$  is a (ordinary) directed homotopy equivalence. Then $f$ is a directed fibre homotopy equivalence.
\end{cor}
\section{The new results}
\subsection{The transition from local to global of dFHE - Dold's type theorems  }
In order to use Corollary \ref{cor.2.5.5} to prove two important Dold type theorems for the directed fibre homotopy equivalences we need some new
notions and preparatory results.
\begin{defn}\label{def.3.1.1}
For a d-space $\underline{B}$ a \emph{directed halo} (\emph{d-halo}) around $\uparrow A\subset \underline{B}$ is a d-subset $\uparrow V$ of $\underline{B}$ if there exist two d-maps $\tau_\alpha:\underline{B}\rightarrow \uparrow \mathbf{I}$, $\alpha\in\{0,1\}$, such that $A\subset \tau_\alpha^{-1}(\alpha)$ and $V$'s complement
$CV\subset \tau_\alpha^{-1}(1-\alpha)$. Obviously $\uparrow A\subset \uparrow V$. If only one of the two functions $\tau_\alpha$ exists
then  we say that $\uparrow V$  is a weak d-halo for $\uparrow A$.
\end{defn}
\begin{rem}\label{rem.3.1.2} 1) In the undirected topology obviously only one map $\tau $ is needed, $\tau=\tau_1$ since for $\tau_0$ we can take $1-\tau_1$.
But even in the case of weak d-halos there can be difficulties since a haloing function must be directed. See the examples below.

2) It is obvious that if $\uparrow V$ is a weak d-halo of $\uparrow A$ in $\underline{B}$, then forgetting those d-structures, $V$ is a halo of $A$ in $B$.
\end{rem}
\begin{ex}\label{ex.3.1.3}
Consider the standard 2-disk $D^2$ centred in $O$  with the boundary $S^1$ and denote $V=D^2\backslash S^1$.

1. Consider for $D^2$ the d-structure $\circlearrowleft D^2$ given by the paths of the standard directed circles with center $O$ and radius $\leq 1$, and
let $\circlearrowleft V$ be the induced d-subspace.

Then we define $\tau_0:\circlearrowleft D^2\rightarrow \uparrow \mathbf{I}$ by $\tau_0(\rho,\theta)=\rho$. For this continuous map we have $\tau_0^{-1}(0)=
\{O\}=: A$ and $\tau_0^{-1}(1)=S^1=CV $. In addition $\tau_0$ is a d-map. Indeed, if $\alpha\in \circlearrowleft D^2$, then $\alpha(t)=(\rho_{\alpha},
\theta(t))$, where $\rho_\alpha$ depends only on $\alpha$. Then $(\tau_0\circ \alpha)(t)=\rho_\alpha$, i.e., it is a constant path and therefore
$\tau_0\circ\alpha\in d\uparrow \mathbf{I}$.

Similarly we define $\tau_1:\circlearrowleft D^2\rightarrow \uparrow \mathbf{I} $ by $\tau_1(\rho,\theta)=1-\rho$. For this we have $\tau_1^{-1}(1)=A$, and
$\tau_1^{-1}(0)=CV$. And as above $\tau_1$ is a d-map since for a d-path $\alpha$, $\tau_1\circ \alpha$ is the constant path $1-\rho_\alpha$. So we verified
that $\circlearrowleft V$ is a d-halo for $A$ in $\circlearrowleft D^2$.

2. Now consider for $D^2$ the d-structure $\uparrow D^2$ given by the paths along radii increasing from the center to the boundary. In this case we can consider also the same map $\tau_0: \uparrow D^2\rightarrow \uparrow \mathbf{I}$. If $\beta\in \uparrow D^2$ , then $\beta(t)=(\rho(t),\theta_\beta)$
and $(\tau_0\circ \beta)(t)=\tau_0(\rho(t),\theta_\beta)=\rho(t)$. Therefore  $\tau_0\circ \beta\in \uparrow \mathbf{I}$, so that $\tau_1$ is in this
case also a d-map.

But a second d-haloing function $\tau_1:\uparrow D^2\rightarrow \uparrow \mathbf{I}$, $\tau_1(\rho,\theta)$,  does not exist. Indeed, it should satisfy the
following condition: a)if $0\leq \rho_1<\rho_2\leq 1$, then $\tau_1(\rho_1,\theta)\leq \tau_1(\rho_2,\theta)$, $(\forall) \theta$ , b) $\tau_1(0,\theta)=1,
(\forall) \theta$, c)$\theta_1(1,0)=0,(\forall)\theta$, which obviously is impossible. Therefore in this case $ \uparrow V$ is no more a d-halo for $A$
in $\uparrow D^2$. But, as in the first case, $V$ is a halo of $A$ in the undirected $D^2$.

\end{ex}
\begin{ex}\label{ex.3.1.4}Consider the unit circle $S^1=\{z\in \mathbb{C}|z(\theta)=cos \theta+i sin \theta, \theta\in [0,2\pi]\}$.
 Consider for $S^1$ the directed strucure $\uparrow S^1$ of d-subspace of $\uparrow \mathbf{R}\times \uparrow \mathbf{R}$. This means that the non trivial paths are included in the arcs a) $\theta\in [\pi/2,\pi]$ and b) $\theta\in [3\pi/2,2\pi]$.

1)Let $\uparrow A,\uparrow V$ be the following d-subspaces of $\uparrow S^1$. $A=\{i\}$, $V=S^1\backslash \{z(\theta)| \theta\in [\pi,3\pi/2]\}$.

We can prove that $\uparrow V$ is a weak d-halo of $\uparrow A$ in $\uparrow S^1$. For this we define $\tau_1:\uparrow S^1\rightarrow
\uparrow \mathbf{I}$  by

\[\tau_1(z(\theta))=\left\{ \begin{array}{ll}
\\\frac{2\theta}{\pi}, & \mbox{if $0\leq \theta\leq \pi/2 $},\\ sin \theta, & \mbox{if $\pi/2\leq \theta\leq \pi $},\\0,&\mbox{if $\pi\leq
\theta\leq 3\pi/2$},\\\frac{1+sin \theta}{2},&\mbox{if $3\pi/2\leq \theta\leq 2\pi$}.\end{array} \right.
\]
For this we have $\tau_1^{-1}(1)=A$, $\tau_1^{-1}(0)\supset CV$. In addition $\tau_1$ is a d-map. Indeed if $\alpha\in$ d$(\uparrow S^1)$ is a non constant  path in the case a), $\alpha(t)=z(\theta(t))$, and if $t_1<t_2$, which implies $\pi/2\leq\theta(t_1)\geq \theta(t_2)\leq \pi$, it follows $(\tau_0\circ \alpha)(t_1)=sin (\theta(t_1))\leq sin (\theta(t_2)=(\tau_1\circ \alpha)(t_2)$, which prove that $\tau_1\circ \alpha\in d(\uparrow \mathbf{I})$. If $\alpha$ is in
the case b)$3\pi/2\leq \theta(t_1)<\theta(t_2)\leq 2\pi$, then $(\tau_1\circ \alpha)(t_1)=\frac{1+sin (\alpha(t_1))}{2} \leq \frac{1+sin (\alpha
(t_2))}{2} =(\tau_1\circ \alpha)(t_2)$,thus that again $\tau_1\circ \alpha\in d(\uparrow \mathbf{I})$. A haloing d-function  $\tau_0$ does not exist since
the relations $\frac{\pi}{2}\leq \theta_1<\theta_2 \Rightarrow \tau_0(\theta_1)\geq \tau_0(\theta_2)$ and $\tau_0(\pi/2)=0,\tau_0(\pi)=1$ are incompatible.

2) Consider instead of $A$ the d-subspace $\uparrow A'=\{z(\theta)\in \uparrow S^1|\pi/2\leq \theta\leq \pi\}$. Then we can prove that $\uparrow V$ is a d-halo of $ \uparrow A'$ in $\uparrow S^1$. Define $\tau'_0:\uparrow S^1\rightarrow \uparrow \mathbf{I}$ by
\[\tau'_0(z(\theta))=\left\{ \begin{array}{ll}
1+\frac{2\theta}{\pi}, & \mbox{if $0\leq \theta\leq \pi/2 $},\\
0, & \mbox{if $\pi/2\leq \theta\leq \pi $},\\\frac{2}{\pi}(\theta-\pi),&\mbox{if $\pi\leq
\theta\leq 3\pi/2$},\\1,&\mbox{if $3\pi/2\leq \theta\leq 2\pi$}.\end{array} \right.
\]
The only fact to note is that $\tau'_0$ is a d-map. But this follows from the fact that for each $\alpha\in d(\uparrow S^1)$ the composition
$\tau'_0\circ \alpha \in d(\uparrow \mathbf{I})$ as a constant path. And $\tau'_1:\uparrow S^1\rightarrow \uparrow \mathbf{I}$ can be defined
by $\tau'_1=1-\tau'_0$.
\end{ex}
\begin{defn}\label{def.3.1.5} A d-map $p:\underline{E}\rightarrow \underline{B}$ has the \emph{directed section extension property} (\emph{d-SEP}) if the following holds.
For every d-subspace $\uparrow A\subset \underline{B}$ and every d-section $s$ over $\uparrow A$ which admits an extension to a d-halo $\uparrow V$ around
$\uparrow A$, there exists a d-extension $S$ over $\underline{B}$, i.e., a d-section $S:\underline{B}\rightarrow \underline{E}$ with $S|A=s$.

\end{defn}
\begin{prop}\label{prop.3.1.6}If a d-map $p:\underline{E}\rightarrow \underline{B}$ is directed dominated by $p':\underline{E}'\rightarrow \underline{B}$, and
$p'$ has the d-SEP, then $p$ has the  d-SEP. In particular, every d-shrinkable map has the d-SEP (since it is d-dominated by $id_{\underline{B}})$.
\end{prop}
\begin{proof}By hypothesis we have d-maps $f:\underline{E}\rightarrow \underline{E}'$, $g:\underline{E}'\rightarrow \underline{E}$ over $\underline{B}$,
$p'f=p,pg=p'$, such that $gf\stackunder{(p)}{\simeq_d}id_{\underline{E}}$. Suppose at first that there exists a d-homotopy
$\varphi:gf\stackunder{(p)}{\preceq_d} id_{\underline{E}}$, i.e., $\varphi(e,0)=gf(e),\varphi(e,1)=e,p\varphi(e,t)=p(e),(\forall) e\in \underline{E},
(\forall) t\in [0,1]$.

Let $\uparrow A\subset \underline{B}$ and $s:\uparrow A\rightarrow \underline{E}$ a d-section, $ps=id_{\uparrow A}$, which admits a d-extension
$s_{\uparrow V}:\uparrow V\rightarrow \underline{B}$, $ps_{\uparrow V}=id_{\uparrow V}, s_{\uparrow V}|\uparrow A=s$, to a d-halo $\uparrow V$, with
$\tau_\alpha:\underline{B}\rightarrow \uparrow \mathbf{I}$ the haloing functions. Denote $\uparrow A'=\uparrow (V\cap \tau_1^{-1}[\frac{1}{2},1])$
and $s'=f\circ s_{\uparrow V}| \uparrow A':\uparrow  A' \rightarrow \underline{E}'$. This is a section over $\uparrow A'$, $\uparrow V$ is a halo around $A'$ and $s'_{\uparrow V}:=f \circ s_{\uparrow V}:\uparrow V \rightarrow \underline{E}'$ is an extension of $s'$. By hypothesis there exists a d-section $S':\underline{B}\rightarrow\underline{E}'$ which extends $s'$.

Now we can define $S:\underline{B}\rightarrow \underline{E}$ by

\[S(b)=\left\{ \begin{array}{ll}
gS'(b), & \mbox{if $\tau_1(b)\leq \frac{1}{2}$},\\
\varphi(s_{\uparrow V}(b),2\tau_1(b)-1), & \mbox{if $\tau_1(b)\geq \frac{1}{2}$}.\end{array} \right.
\]
This is a well-defined d-map and $pS(b)=pgS'(b)=p'S'(b)=b$, if $\tau_1(b)\leq \frac{1}{2}$ , and $pS(b)=p\varphi(s_{\uparrow V}(b),2\tau_1(b)-1)=ps_{\uparrow V}(b)=b$, if
$\tau_1(b)\geq \frac{1}{2}$, i.e., $S$ is a section. Then if $b\in A$, then $\tau_1(b)=1$ and $S(b)=\varphi(s_{\uparrow V}(b),1)
=s_{\uparrow V}(b)=s(b)$, i.e., $S|\uparrow A=s$.

Now if we have the case $id_{\underline{E}}\stackunder{(p)}{\preceq_d}gf$, by a d-homotopy $\psi$ , $\psi(e,0)=e,\psi(e,1)=gf(e),p\psi(e,t)=p(e)$,
then we repeat the steps above by considering the d-section $\overline{S'}:\underline{B}\rightarrow \underline{E}' $ for which $\overline{S'}|
\uparrow (V\cap\tau_0^{-1}[0,\frac{1}{2}])=(f\circ s_{\uparrow V})|\uparrow (V\cap\tau_0^{-1}[0,\frac{1}{2}])$ and define $\overline{S}:\underline{B}\rightarrow \underline{E}$ by
\[\overline{S}(b)=\left\{ \begin{array}{ll}
\psi(s_{\uparrow V}(b),2\tau_0(b)), & \mbox{if $\tau_0(b)\leq \frac{1}{2} $},\\
g\overline{S'}(b), & \mbox{if $\tau_0(b)\geq \frac{1}{2} $}.\end{array} \right.
\]
Another case is when $\kappa:gf\stackunder{(p)}{\preceq_d}h$ and $\theta:id_{\underline{E}}\stackunder{(p)}{\preceq_d}h.$ Consider the composition
$\widetilde{s}=h\circ s:\uparrow A\rightarrow \underline{E}$. This is a section, $p\circ (h\circ s)=(p\circ h)\circ s=p\circ s=id_{\uparrow A}$. By the
notations from the first case $\widetilde{s}'=fhs_{\uparrow V}|\uparrow A': \uparrow A'\rightarrow \underline{E'} $ is also a section for $p'$
having as an extension the d-map  $\widetilde{s}'_{\uparrow V}=fhs_{\uparrow V}:\uparrow V \rightarrow E'$ . Denote by
$\widetilde{S}':\underline{B}\rightarrow \underline{E}'$ a d-section which extends $\widetilde{s}'$. Define $\widetilde{S}:\underline{B}\rightarrow
\underline{E}'$ by
\[\widetilde{S}(b)=\left\{ \begin{array}{ll}
g\widetilde{S}'(b), & \mbox{if $\tau_1(b)\leq \frac{1}{2}$},\\
\varphi(hs_{\uparrow V}(b),2\tau_1(b)-1), & \mbox{if $\tau_1(b)\geq \frac{1}{2}$}.\end{array} \right.
\]
This is well defined, it is a section and $\widetilde{S}|\uparrow A'=h\circ s_{\uparrow V}$.
Now define $S:\underline{B } \rightarrow \underline{E}$ by

\[S(b)=\left\{ \begin{array}{ll}
\theta(s_{\uparrow V}(b),2\tau_0(b)), & \mbox{if $\tau_0(b)\leq \frac{1}{2} $},\\
\widetilde{S}(b), & \mbox{if $\tau_0(b)\geq\frac{1}{2} $}.\end{array} \right.
\]
This is well defined and it is a section for $p$. Then if $b\in \uparrow A$, then $\tau_0(b)=0$, such that $S(b)=\theta(s_{\uparrow V}(b),0)=s_{\uparrow V}(b)=s(b)$.

From these three situations it can be deduced how to proceed in other cases.
\end{proof}
\begin{defn}\label{def.3.1.7}  A covering $\{V_\lambda\}_{\lambda\in \Lambda}$ of a directed space $\underline{B}$ is called \emph{d-numerable}
if the following condition is satisfied:  there exists a directed locally finite partition of unity $\{\pi_\gamma:
\underline{B} \rightarrow \uparrow \mathbf{I}\}_{\gamma\in \Gamma}$ such that every set $\pi_\gamma^{-1}(0,1]$ is contained in some $V_\lambda$.
\end{defn}

The proof of Th.2.7 (Section Extension Theorem) from \cite{Dold1} can be smoothly adapted, having as model the proof of Proposition \ref{prop.3.1.6},
to the proof of the following proposition.
\begin{prop}\label{prop.3.1.8}Let $p:E\rightarrow B$ be a d-map. If there exists a d-numerable covering $\{V_\lambda\}$ of $\underline{B}$ such that $p$ has
the d-SEP over each $\uparrow V_\lambda$, then $p$ has the d-SEP.
\end{prop}
The following proposition is analogous to Prop. 3.1 in \cite{Dold1} from the undirected case. But to make a proof accessible to the author, a
restriction has to be imposed given by the following definition.
\begin{defn}\label{def.3.1.9}A d-map $p:\underline{E}\rightarrow \underline{B}$ is called \emph{saturated} if $\alpha\in d\underline{E}$ if and only if
$p\circ \alpha\in d\underline{B}$.
\end{defn}

\begin{prop}\label{prop.3.1.10} Consider the following properties of a d-map $p:\underline{E }\rightarrow \underline{B}$ .

(a) For a d-map  $\alpha:\underline{X}\rightarrow \underline{B}$, the induced d-map $p_\alpha:\uparrow E_\alpha:=\{(e,x)\in \underline{E}\times
\underline{X}|p(e)=\alpha(x)\}\rightarrow \underline{X}, p_\alpha(e,x)=x$, has the d-SEP (in particular $p=p_{id}$).

(b) Given $\overline{F}:\underline{X}\rightarrow \underline{B}$, a d-halo $\uparrow V$ around $\uparrow A\subset \underline{X}$, and $f:\uparrow V
\rightarrow \underline{E} $ with $pf=\overline{F}|\uparrow V$(a partial lift of $\overline{F}$) there exists $F:\underline{X}\rightarrow \underline{E}$
with $F|A=f|A$, and $ pF=\overline{F}$ ( a lift of $\overline{F}$ ).

(c) Given $\overline{F}:\underline{X}\rightarrow \underline{B}$, $\uparrow A\subset \underline{X}$, $ f:\uparrow A\rightarrow \underline{E}$ with $pf=
\overline{F}|\uparrow A $, there exists a lift $F:\underline{X}\rightarrow \underline{E}$ of $\overline{F}$ with $F|\uparrow A\stackunder{(p)}{\simeq_d}f$.

(d) $p:\underline{E}\rightarrow \underline{B}$ is shrinkable.

Then we have:

I. (a)$\Leftrightarrow$ (b);

II. (c)$\Rightarrow$ (b)

III. (d) $\Rightarrow$ (c);

IV. If $p$ is a saturated d-map then (b)$\Rightarrow $(d).
\end{prop}
\begin{proof}

I. (a)$\Rightarrow$ (b). Given $\overline{F}$, $\uparrow A\subseteq\uparrow V $, $f$ as in (b) and consider the induced d-map $p_{\overline{F}}:
\underline{E}_{\overline{F}}=\{(e,x)\in \underline{E}\times \underline{X}|p(e)=\overline{F}(x)\}\rightarrow \underline{X}$, $p_{\overline{F}}(e,x)=x$.
We can consider the d-map $s:V\rightarrow \underline{E}_{\overline{F}}$ defined by $s(v)=(f(v),v)$, since $pf(v)=\overline{F}(v)$. This is a section for
$p_{\overline{F}}$ over $\uparrow V$, $p_{\overline{F}}s(v)=v$. Then by (a) there exists a section $S:\underline{X}\rightarrow \underline{E}_{\overline{F}}$, $p_{\overline{F}}\circ S=id_{\underline{X}}$. If $S(x)=(F(x),\mathcal{X}(x))$, then $pF(x)=\overline{F}(\mathcal{X}(x))$, $\mathcal{X}(x)=
id_{\underline{X}}(x)=x$, such that $pF(x)=\overline{F}(x)$.

(b)$\Rightarrow$ (a) Suppose given an induced d-map $p_\alpha:\underline{E}_\alpha \rightarrow \underline{X}$ and let $\uparrow A\subset \underline{X}$
$s:\uparrow A \rightarrow \underline{E}_\alpha $ a section of $p_\alpha$ over $\uparrow A$, $\uparrow V$ a weak d-halo around of $\uparrow A$ and
$s_{\uparrow V}: \uparrow V \rightarrow \underline{E}_\alpha$ and extension of $s$, $p_\alpha s_{\uparrow V}=id_{\uparrow V}, s_{ \uparrow V}|
\uparrow A=s$. Now we place in terms of (b) with $\overline{F}=\alpha:\underline{X}\rightarrow \underline{B}$, and if $s_{\uparrow V}(v)=(\mathcal{E}(v),
v)$, $f:=\mathcal{E}:\uparrow V\rightarrow \underline{E}$. For these we have $pf(v)=p\mathcal{E}(v)=\alpha(v)=\overline{F}(v)$. Then there exists
$F:\underline{X}\rightarrow \underline{E}$, with $F|\uparrow A=f|\uparrow A=\mathcal{E}|\uparrow A$ and $pF=\alpha$. Now we define $S:\underline{X}
\rightarrow \underline{E}_\alpha$ by $S(x)=(F(x),x)$ and this is a section which extends $s$.

II. (c)$\Rightarrow$ (b): Given $F$, $\uparrow A\subset \uparrow V$ and  $f$, as in (b), and assuming (c), there exists a lift $F':\underline{X}\rightarrow
\underline{E}$ of $F$ and such that $F'|V\stackunder{(V)}{\simeq_d}f$. Suppose at first $D: F'|\uparrow V\stackunder{(V)}{\preceq_d}f$. If $\tau_1:
\underline{X}\rightarrow \uparrow \mathbf{I}$ with $\tau_1|A=1,\tau_1|CV=0$, define $F:\underline{X}\rightarrow \underline{E}$
\[F(x)=\left\{ \begin{array}{ll}
F'(x), & \mbox{if $\tau_1(x)\leq \frac{1}{2} $},\\
D(x,2\tau_1(x)-1), & \mbox{if $\tau_1(x)\geq\frac{1}{2} $}.\end{array} \right.
\]
This is well defined. If $x\in \uparrow A, F'(x)=f(x),D(x,1)=f(x)$, i.e., $F| \uparrow A=f$. Then if $\tau_1(x)\leq \frac{1}{2}$,
$pF(x)=pF'(x)=\overline{F}(x)$ and if $\tau_1(x)\geq \frac{1}{2}$, $pF(x)=pD(x,2\tau_1(x)-1)=\overline{F}(x)$. If we have
$\widetilde{D}:f \stackunder{(V)}{\preceq_d}F'| \uparrow V$,consider $\tau_0:\underline{X}\rightarrow \uparrow \mathbf{I}$ with $\tau_0|\uparrow A=0,
\tau_0|CV=1$ and define $\widetilde{F}:\underline{X}\rightarrow \underline{E}$ by

\[F(x)=\left\{ \begin{array}{ll}
\widetilde{D}(x,2\tau_0(x)), & \mbox{if $\tau_0(x)\leq \frac{1}{2} $},\\
F'(x), & \mbox{if $\tau_0(x)\geq \frac{1}{2} $}.\end{array} \right.
\]
If $\kappa:F'|\uparrow V \stackunder{(V)}{\preceq_d}h$ and $\theta: f\stackunder{(V)}{\preceq_d}h$, as above we obtain
$\widehat{F}:\underline{X}\rightarrow\underline{E}$ a lift of $\overline{F}$ and $\widehat{F}|\uparrow A=h|\uparrow A$ and then define

\[F(x)=\left\{ \begin{array}{ll}
\theta(x,2\tau_0(x)), & \mbox{if $\tau_0(x)\leq \frac{1}{2} $},\\
\widehat{F}(x), & \mbox{if $\tau_0(x)\geq \frac{1}{2} $}.\end{array} \right.
\]
From these three situations it can be deduced how to proceed in other cases.

III.For the implication (d)$\Rightarrow$ (c) the proof is the same as in the undirected case since no inverse paths (homotopies)(\cite{Dold1},
pp.231-232) are used.

IV. Apply (b) for $\underline{X}=\underline{B}$ $A=V=\emptyset$, and $\overline{F}=id_{\underline{B}}$ Then there exists
$S:\underline{B}\rightarrow \underline{E}$ with $p\circ S=id_{\underline{B}}$. We can prove that $S$ satisfies also the relation $S\circ p
\stackunder{(p)}{\simeq_d}id_{\underline{E}}$. For this purpose we consider $\underline{X}=\underline{E}\times \uparrow \mathbf{I}$,
$\overline{F}:\underline{X } \rightarrow \underline{B}$ ,$\overline{F}(e,t)=p(e)$, $\uparrow A=\underline{E}\times \{0\}\cup \underline{E}\times \{1\}$,
$\uparrow V=\underline{E}\times \uparrow[0,\frac{1}{2})\cup \underline{E}\times \uparrow(\frac{1}{2},1]$ , and $f:\uparrow V\rightarrow \underline{E}$
given by
\[f(e,t)=\left\{ \begin{array}{ll}
x, & \mbox{if $t<\frac{1}{2} $},\\
Sp(x), & \mbox{if $t>\frac{1}{2} $}.\end{array} \right.
\]
Now by Remark \ref{rem.3.1.2}, 2) and \cite{Dold1}, Prop.3.1, (b)$\Rightarrow$ (d), we deduce that there exists $F:E\times [0,1]\rightarrow E$ such that
$p\circ F=\overline{F }$ with $F|A=f|A$. This homotopy satisfies the conditions $F(e,0)=f(e,0)=e$, $F(e,1)=f(e,1)=Sp(e),pF(e,t)=\overline{F}(e,t)
=p(e),(\forall)e\in E, (\forall) t\in [0,1]$. Therefore $F:id_B \stackunder{(p)}{\simeq}Sp$. But since $F|A=f|A$ and $f$ is a d-map, it follows that $F$ is directed as in the first argument. Then from the relation $pF(e,t)=p(e)$ it follows that for a d-path $\alpha$  in $\underline{E}\times \uparrow
\mathbf{I}$, $\alpha(t)=(e(t),i(t))$, the composition $F\circ \alpha$ is a lift of d-pat $ e(t)$. This, by hypothesis implies that $F\circ \alpha
\in d(E\times \uparrow \mathbf{I})$ . Therefore $F$ is a d-homotopy, such that $F:id_{\underline{B}} \stackunder{(p)}{\simeq_d} S\circ p$.

\end{proof}
\begin{cor}\label{cor.3.1.11}
If $p:\underline{E }\rightarrow \underline{B}$ is a saturated d-map and d-shrinkable over each d-set $\uparrow V_\lambda$ of a d-numerable covering
$\{V_\lambda\}$ of $\underline{B}$, then $p$ is d-shrinkable.
\end{cor}
\begin{proof}Under the imposed conditions we have in Proposition \ref{prop.3.1.10} (a)$\Leftrightarrow$ (d), such that it is sufficient to show that for every d-map
$\alpha:\underline{X}\rightarrow \underline{B}$, the induced d-map $p_\alpha:\uparrow E_\alpha\rightarrow \underline{X}$ has the d-SEP. Now,
$\{\alpha^{-1}(V_\lambda)\}$ is a d-numerable covering of $\underline{X}$, and $p_\alpha$ is d-shrinkable over $\uparrow \alpha^{-1}(V_\lambda)$. Indeed,
because $p$ is d-shrinkable over $\uparrow V_\lambda$ there is a section $s:\uparrow V_\lambda \rightarrow\uparrow p^{-1}(V_\lambda)$, $p\circ s=
id$ , $p\circ s \stackunder{(p)}{\simeq_d}id $ and we can define $s':\uparrow \alpha^{-1}(V_\lambda)\rightarrow \uparrow
(\alpha\circ p_\alpha)^{-1}(V_\lambda)$ by $s'(x_\lambda)=(s(\alpha(x_\lambda)),x_\lambda)$, for which $p_\alpha\circ s'=id$ and $(s'\circ p_\alpha)
(e,x)=s'(x)=(s(\alpha(x)),x)=(sp(e),x)$, and then there is immediately the relation $s'\circ p_\alpha\stackunder{(p_\alpha)}{\simeq_d} id $. Now by
Prop. 3.1.9,(d)$\Rightarrow $ (a), $p_\alpha$ has d-SEP over $\uparrow\alpha^{-1}(V_\lambda)$, and then we apply Prop. \ref{prop.3.1.8} and Prop.
\ref{prop.3.1.10}, (a)$\Rightarrow$ (d).
\end{proof}

Let $p:\underline{E}\rightarrow  \underline{B}$ , $p':\underline{E}' \rightarrow \underline{B}$ directed maps and $f:E'\rightarrow \underline{E}$ a
d-map over $\underline{B}$, $p\circ f=p'$. Consider the following d-subspace $\underline{R}$ of $\underline{E}'\times \underline{E}^{\uparrow
\mathbf{I}}$ defined by
$$\underline{R}=\{(e',w)|pw(I)=p'(e'),w(1)=f(e')\},$$
and the map
$$q:\underline{R}\rightarrow \underline{E}, q(e',w)=w(0).$$
\begin{lem}\label{lem.3.1.12} Suppose that the following conditions are satisfied: 1) $p$ is saturated, 2)  all fibres of $p$ and $p'$ are symmetric directed spaces, 3) $f$ is a directed fibre homotopy equivalence. Then $q$ is a d-shrinkable d-map.
\end{lem}
\begin{proof}We prove at first that $q$ is a d-map. Let $\alpha$ be a d-path in $\underline{R}$, $\alpha(t)=(e'(t),w(t))$, with $e'(.)\in d(\underline{E}')$ and
$w(.)\in d(\underline{E}^{\uparrow \mathbf{I}})$. Then $(q\circ \alpha)(t)=w(t)(0)$ and $p(q\circ \alpha)(t)=p(w(t)(0))=p'(e'(t))$, i.e., $p(q\circ \alpha)=
p'\circ e'(.)$. And since $p'\circ e'(.))\in d(\underline{B})$ and $p$ is saturated, we deduce that $q\circ \alpha\in d(\underline{E})$.

Then the condition on the fibres of $p$ and $p'$ allows us, as we shall see, to paraphrase the proof given in \cite{Dold1}, p.233 (Proof of Lemma 3.4). However we give this proof with the necessary specifications from the beginning and some details that can help the reader.

It is sufficient to prove that there exists a d-map $\sigma:\underline{E}\rightarrow \underline{R}$ such that $q\circ \sigma=id_{\underline{E}}$ and
$\sigma\circ q\stackunder{(q)}{\simeq_d} id_{\underline{R}}$ .

To define $\sigma$, use the hypothesis. Let $f':\underline{E}\rightarrow \underline{E}'$ denote a directed fiber homotopy inverse of $f$, $f\circ f'\stackunder{(p')}{\simeq_d}
id_{\underline{E}}$ and $f'\circ f\stackunder{(p)}{\simeq_d}id_{\underline{E}'}$. Due to the hypothesis on the fibres, it is sufficient to consider the cases
$id_{E'}\stackunder{(p')}{\preceq_d}f'f$ by a vertical d-homotopy $\varphi:\underline{E}'\times \uparrow \mathbf{I}\rightarrow \underline{E}'$, satisfying
$\varphi(e',0)=e',\varphi(e',1)=f'f(e'),p'\varphi(e',t)=p'(e'),(\forall) e'\in \underline{E}',t\in [0,1]$, and $id_{E}\stackunder{(p)}{\preceq_d} ff'$
by a vertical d-homotopy $\psi:\underline{E}\times \uparrow \mathbf{I}\rightarrow \underline{E}$, satisfying $\psi(e,0)=e,\psi(e,1)=ff'(e), p\psi(e,t)=
p(e),(\forall)e\in \underline{E},t\in [0,1] $.

Now define
$$\sigma:\underline{E}\rightarrow \underline{R}, \sigma(e)=(f'(e),\psi_e).$$
This is well defined since $\psi_e(1)=\psi(e,1)=f(f'(e))$ and $p\psi_e(t)=p\psi(e,t)=p(e)=p'(f'(e))$. And since $f'$ and $\psi$ are d-maps, it follows that
$\sigma$ is a d-map. For this we have $(q\sigma)(e)=q(f'(e),\psi_e)=\psi_e(0)=\psi(e,0)=e $, i.e., $q\circ \sigma=id_{\underline{E}}$. Then
$$(\sigma q)(e',w)=\sigma(w(0))=(f'(w(0)),\psi_{w(0)}). $$
In the proof that $\sigma\circ q\stackunder{(q)}{\simeq_d} id_{\underline{R}}$,for a pair $(e',w)\in \underline{R}$, the following notations are used:
1) for $\tau\in [0,1], _\tau w$, defined by $_\tau w(t)=w(t\tau)$, is obviously a d-path; 2) $w^{-}$ , defined by $w^{-}(t)=w(1-t)$ , is a d-path since
$w(I)\subset p^{-1}(p(e'))$; 3) $^\tau w$ defined by $^\tau(t)=w(1-\tau+t\tau)$, which is $^\tau w=(_\tau(w^{-}))^{-}$, so that this is also a
d-path. Then for $e'\in \underline{E}'$, $\varphi_{e'}$ is the d-path defined by $\varphi_{e'}(t)=\varphi(e',t)$ for which $\varphi_{e'}(I)\subset
p'^{-1}(p'(e'))$ so that $\varphi_{e'}^{-}$ is a d-path. Similarly, for $e\in \underline{E}$,$\psi_e(I)\subset p^{-1}(p(e)$, so that $\psi_e^{-}$ is a
d-path. Finally, for $\psi_{1-\tau}:\underline{E}\rightarrow \underline{E}$, defined by $\psi_{1-\tau}(e)=\psi (e,1-\tau)$, and for a d-path $\alpha\in
d(\underline{E})$ we have $(\psi_{1-\tau}\circ \alpha)(t)=\psi(\alpha(t),1-\tau)=\psi_{\alpha(t)}(1-\tau)=(\psi_{\alpha(t)})^{-}(\tau)$ , which shows that
$\psi_{1-\tau}$ is a d-map. Now we proceed to the construction of a series of directed vertical homotopies from $id_{\underline{R}}$ to $\sigma\circ q$.

At first we consider the d-homotopy $C:\underline{R}\times \uparrow \mathbf{I}\rightarrow \underline{R} $ defined by $C((e',w),t)=
(e',_{\frac{t+1}{2}}(w\ast c))$ , where $c$ is the constant path. This is well defined since $_{\frac{t+1}{2}}(w\ast c)(1)=(w\ast c)(\frac{t+1}{2})=
c(t)=w(1)=f(e')$, and $p_{\frac{t+1}{2}}(w\ast c)(\tau))=p'(e')$. Then $qC((e'w),t)=_{\frac{t+1}{2}}(w\ast c)(0)=w(0)=q(e'w)$, so that $C$ is a
vertical d-homotopy, with $C_0=id_{\underline{R}}$ and $C_1((e',w)=(e',w\ast c)$.
\begin{equation}
id_{\underline{R}}\stackunder{(q)}{\simeq_d}C_1
\end{equation}
Now consider the following d-map $u:\underline{R}\rightarrow \underline{E}'^{\uparrow \mathbf{I}}$ , defined by
$$ u(e',w)=\varphi_{e'}\ast (f'w^{-})\ast (f'\psi_{w(0)})\ast(f'ff'w)\ast(f'f\varphi_{e'}^{-})\ast (f'w^{-}).$$
The ends of this d-path are $u(e',w)(0)=e'$ and $u(e',w)(1)=(f'w^{-})(1)=f'(w(0))$, and $p'u(e',w)(I)=p'(e')$
This suggests to define the next d-homotopy
$$ U:\underline{R}\times \uparrow \mathbf{I} \rightarrow \underline{R},U((e'w),\tau)=(u(e',w)(\tau),w\ast (f\circ(_\tau u(e',w)).$$
This is well defined since $p(w\ast(f\circ _\tau u(e',w))(I)=pw\ast pf(_\tau u(e',w))(I)=(pw\ast p'_\tau u(e'w))(I)=p'(e')$ and
$(w\ast (f\circ _\tau u(e',w))(1)=f(_\tau u(e',w)(1))=f(u(e',w)(\tau))$. Then $qU((e',w),\tau)=(w\ast (f\circ _\tau u(e',w)))(0)=w(0)p'(e')=q(e'w)$.
And $U((e',w),0)=(u(e',w)(0),w\ast f\circ c)=(e',w\ast c)=C_1(e',w)$. Therefore
\begin{equation}
U:C_1\stackunder{(q)}{\simeq_d}U_1,
\end{equation}
with
$$ U_1(e',w)=(f'(w(0)),w\ast(f\circ u(e',w)).$$
We can write $U$ in a more convenient form, namely : $w\ast fu(e'w)=[w\ast(f\varphi_{e'})\ast(ff'w^{-})]\ast(ff'\psi_{w(0)})\ast (ff')\circ
[(ff'w)\ast(f\varphi_{e'}^{-})\ast w^{-}]$, and for $v:=(ff'w)\ast (f\varphi_{e'}^{-})\ast w^{-}$,we have
$$ U_1(e',w)=(f'(w(0)),v\ast(ff'\psi_{w(0)})\ast (ff'v^{-}).$$
Now by two homotopies similar $C$ we deduce that
\begin{equation} U_1\stackunder{(q)}{\simeq_d}{U'}_1,
\end{equation}
$$U'_1(e',w)=(f'(w(0)),v\ast c\ast(ff'\psi_{w(0)})\ast(ff'v^{-})\ast c).$$
Now define $U'':\underline{R}\times \uparrow \mathbf{I} \rightarrow \underline{R}$, by
$$U''((e',w),\tau)=(f'(w(0)),\theta:=v\ast _\tau(\psi_{w(0)}^{-})\ast \psi_{1-\tau}(\psi_{w(0)}\ast v^{-})\ast ^\tau(\psi_{w(0)})).$$
This is well defined since $\theta(1)=^\tau(\psi_{w(0)})(1)=\psi_{w(0)}(1-\tau+\tau)=\psi_{w(0)}(1)=\psi(w(0),1)=f(f'(w(0))$ and
$p\theta(I)=p'(e')=p(w(0))=p'(f'(w(0))$. Moreover, $qU''((e',w),\tau)=\theta(0)=v(0)=w(0)=q(e'w)$, so that $U''$ is a vertical d-homotopy, with
$U''((e',w),0)=(f'(w(0)),v\ast c\ast ff'(\psi_{w(0)} \ast v^{-})\ast c)=U'_1(e',w)$, and therefore
\begin{equation}U'':U'_1\stackunder{(q)}{\simeq_d}U''_1,
\end{equation}
where
$$U''_1(e',w)=(f'(w(0)),(v\ast \psi_{w(0)}^{-})\ast (v\ast \psi_{w(0)}^{-})^{-}\ast \psi_{w(0)}).$$
Finally, if we consider the d-map $U_1{'''}:\underline{R} \rightarrow \underline{R}$, defined by
$$U'''_1(e',w)=(f'(w(0)),c\ast \psi_{w(0)}),$$
it is immediate (by the usual homtopies for paths, $\alpha\ast\alpha^{-}\simeq c $ rel $\partial I$ and $\alpha\simeq \alpha\ast c$ rel $\partial I$) that,
\begin{equation}
\sigma\circ q\stackunder{(q)}{\simeq_d}U'''_1, U''_1\stackunder{(q)}{\simeq_d}U'''_1.
\end{equation}
In conclusion, by the relations 7.1-7.5, we have
$$\sigma\circ q\stackunder{(q)}{\simeq_d}U'''_1\stackunder{(q)}{\simeq_d}U''_1\stackunder{(q)}{\simeq_d}U'_1\stackunder{(q)}{\simeq_d}C_1
\stackunder{(q)}{\simeq_d}id_{\underline{R}}.$$
\end{proof}
\begin{thm}\label{thm.3.1.13} Let $p:\underline{E}\rightarrow \underline{B}$ , $p':\underline{E}'\rightarrow \underline{B}$ ,
$f:\underline{E}'\rightarrow \underline{E}$ denote d-maps satisfying the following properties: 1) $p$ and $p'$ are saturated d-maps and having both all
fibres symmetric directed spaces, 2) $f$ is over $\underline{B}$, i.e., $p\circ f=p'$, 3) $f$ is a directed fibre homotopy equivalence over each d-set $\uparrow V_\lambda$ of a d-numerable covering $\{V_\lambda\}$ of $\underline{B}$. Then $f$ is directed fibre homotopy equivalence.
\end{thm}

\begin{proof}
We denote by $p_\lambda=p|\uparrow p^{-1}(V_\lambda) $ and $p'_\lambda=p'|\uparrow p'^{-1}(V_\lambda)$ and by $q_\lambda$ the
d-map $q$ from Lemma \ref{lem.3.1.12} defined for the d-maps $p_\lambda$ and $p'_\lambda$. We have that $q_\lambda$ is d-shrinkable. Then
$\{p^{-1}(V_\lambda)\}$ is a d-numerable covering of $\underline{E}$ and $q_\lambda$ is the restriction of $q$ . To apply Lemma \ref{lem.3.1.12}
we need to verify  that $q$ is saturated. Suppose that $\alpha(t)=(e'(t),w(t)$ is a path in $R$ such that $q\circ \alpha$ is a d-path in $\underline{E}$,
i.e., $w(.)(0)$ is a d-path in $\underline{E}$. It follows that $p\circ w(.)(0)$ is a d-path in $\underline{B}$. But $p(w(t)(0))=p'(e'(t))$, which means
that $p\circ w(.)(0)=p'\circ e'(.)$. Therefore $p'\circ e'(.)$ is a d-path in $\underline{B}$ . Then because $p'$ is saturated, we deduce that $e'(.)
\in d(\underline{E}')$. By the relation $p'(e'(t))=p(w(t)(I)$, we deduce that $w(t)\in d(\underline{E})$, $(\forall) t\in [0,1]$. The relation $w(t)(1)=
f(e'(t)$ shows that $w(.)(1)=f\circ e'(.)$ and therefore $w(.)(1)$ is directed. Finally from the equality $p(w(t)(t'))=p(w(t)(1)$ $(\forall)t $, we conclude that $w\in d(\underline{E}^{\uparrow \mathbf{I}})$, and therefore $\alpha\in d(\underline{R})$. Now applying Corollary \ref{cor.3.1.11} we obtain that $q$ is d-shrinkable. Then by Prop. \ref{prop.3.1.6} $q$ has the d-SEP.

Consider a section $S:\underline{E}\rightarrow \underline{R}$ for $q$ written as $S=(f',\theta)$ with $f':\underline{E}\rightarrow \underline{E}'$ a d-map and $\theta:\underline{E}\rightarrow \underline{E}^{\uparrow \mathbf{I}}$ also a d-map. From the relations $p'(f'(e))=p(\theta(e)(I))$, $qS=id_{\underline{E}}$, i.e., $\theta(e)(0)=e$, we deduce that $p'\circ f'=p$ , i.e., $f'$ is a d-map over $\underline{B}$. Then we can define $\Theta:
\underline{E}\times \mathbf{I}\rightarrow \underline{E}$ by $\Theta(e,t)=\theta(e)(t)$. This is a vertical d-homotopy, $p\Theta(e,t)=p\theta(e)(t)=p'f'(e)=
p(e)$, with $\Theta(e,0)=\theta(e)(0)=e$ and $\Theta(e,1)=\theta(e)(1)=f(f'(e))$. Therefore $\Theta:id_{\underline{E}}\stackunder{(p)}{\simeq_d}ff'$. This
relation implies that, over $\uparrow V_\lambda$, the d-map $f'$ is a directed fibre homotopy inverse to $f$. In particular it is a directed fibre homotopy equivalence. Then we can apply the above argument to $f'$ instead of $f$, and find $f'':\underline{E}'\rightarrow \underline{E}$ with
$f'f''\stackunder{(p)}{\simeq_d}id_{\underline{E}'}$, hence $f'f\stackunder{(p)}{\simeq_d} (f'f)(f'f'')=f'(ff')f''\stackunder{(p)}{\simeq_d}f'f''
\stackunder{(p)}{\simeq_d}id_{\underline{E}'}$.
\end{proof}
\begin{rem}\label{rem.3.1.14}Under the same conditions as in Theorem \ref{thm.3.1.13}, let $\uparrow V$ be a d-halo of $\uparrow A\subset \underline{B}$. Suppose
that the restriction of $f$, $f_V:\uparrow p'^{-1}(V)\rightarrow \uparrow p^{-1}(V)$, is a directed fiber homotopy equivalence, with $f^{-}_V:\uparrow
p^{-1}(V)\rightarrow \uparrow p'^{-1}(V)$ as a d-homotopy inverse. Then $f^{-}_A$ and the restrictions of all d-homotopies which determine the vertical
d-homotopy equivalence $id_{\uparrow p^{-1}(V)}\stackunder{(V)}{\simeq_d}f_V\circ f^{-}_V$ can be extended over all $\underline{B}$. This is obtained if
in the proof of Theorem \ref{thm.3.1.13} we choose (using the d-SEP) the section $S=(f',\theta)$ such that $f'|\uparrow p^{-1}( A)= f_A^{-}$ and $\Theta
(z,t)=\theta(z)(t)$ to be over $\uparrow A\times \uparrow \mathbf{I}$ a certain d-homotopy.

\end{rem}
\begin{thm}\label{thm.3.1.15}Let $\underline{B}$ a directed space which admits a directed numerable covering $\{V_\lambda\}$ such that the inclusion d-map
$\uparrow V_\lambda\rightarrow \underline{B}$ is directed null-homotopic for every $\lambda$. Let $p:\underline{E}\rightarrow \underline{B}$ ,
$p':\underline{E}'\rightarrow \underline{B}$ be saturated directed weak fibrations and having both all fibres symmetric directed spaces. Then a d-map
$f:\underline{E}\rightarrow \underline{E}'$ over $\underline{B}$, $p'\circ f=p$, is a directed fibre homotopy equivalence if and only if the restriction
of $f$ to every fibre, $f_b:\uparrow p^{-1}(b)\rightarrow \uparrow p'^{-1}(b)$, $b\in \underline{B}$ is an (ordinary) directed homotopy equivalence.
\end{thm}
\begin{proof}
If $f$ is a directed fibre homotopy equivalence then it is immediate that  $f_b$ is a directed homotopy equivalence for every $b\in \underline{B}$. For the
converse, by Theorem \ref{thm.3.1.13}, it is enough to show that $f_V:\uparrow p^{-1}(V)\rightarrow \uparrow p'^{-1}(V)$ is a directed fibre homotopy equivalence. But in the given conditions this follows from Corollary \ref{cor.2.5.5}. This completes the proof of the theorem.
\end{proof}

\subsection{Transition from local to global of dWCHP-a tom Dieck-Kamps-Puppe type theorem}

To prove a theorem of the type local-global for the d-WCHP, we follow the approach from \cite{Dieck}, Prop.9.5, for the undirected case.
\begin{prop}\label{prop.3.2.1}Let $\mathcal{U}=\{U_j|j\in J\}$ be a d-numerable covering of the product $X\times \uparrow \mathbf{I}$. Then there exists a d-numerable covering $\{V_k| k\in K\}$ of $\underline{X}$ and a family $\{\varepsilon_k|k\in K\}$ of positive real numbers, such that for $t_1,t_2\in [0,1],0 \leq t_2-t_1<\varepsilon_k$ there is an index $j\in J$ for which $V_k\times [t_1,t_2]\subset U_j$.
\end{prop}
\begin{proof}The proof is the same as that of Theorem 8.3 from \cite{Dieck} and we just point out some aspects of d-topology. Denote by $\{u_j|j\in J\}$
a directed locally finite partition of unity such that $U_j=u_j^{-1}((0,1])$, $j\in J$. For each r-tuple $k=(j_1,...,j_r)\in J^r$ define the continuous map $v_k:X\rightarrow I$ by
$$ v_k(x)=\stackunder{i=1}{\stackrel{r}{\prod}}Min(u_{j_i}(x,t)|t\in [\frac{i-1}{r+1},\frac{i+1}{r+1}]).$$
We can prove that this is a directed map $v_k:\underline{X}\rightarrow \uparrow \mathbf{I}$. Indeed suppose that $\alpha\in d\underline{X}$ and $t_1<
t_2$. Then because $u_{j_i}:\underline{X}\times \uparrow\mathbf{I}\rightarrow \uparrow \mathbf{I}$ is a d-map $u_{j_i}(\alpha(t_1),t)\leq u_{j_i}
(\alpha(t_2)$, $(\forall)t\in [\frac{i-1}{r+1},\frac{i+1}{r+1}]$, which implies $Min(u_{j_i}(\alpha(t_1),t)\leq Min(u_{j_i}(\alpha)t_2),t)$,
$(\forall)t\in [\frac{i-1}{r+1},\frac{i+1}{r+1}]$. Therefore $(v_k\circ \alpha)(t_1)\leq (v_k\circ \alpha)(t_2)$ which proves that $v_k$ is a directed
map. Consider $K=\stackunder{r=1}{\stackrel{\infty}{\bigcup}}J^r$. It proves (\cite{Dieck},pp.144-145) that $\{V_k:=v_k^{-1}((0,1])|k\in K\}$ is a (d-)numerable covering of $\underline{X}$, and for $\varepsilon_k=\frac{1}{2r}$ for $k=(j_1,...,j_r)$, the conclusion of the theorem is satisfied.
\end{proof}
\begin{prop}\label{prop.3.2.2}Let $p:\underline{E}\rightarrow \underline{B}$ be a d-map and $0<\varepsilon<1$. The following properties of $p$ are equivalent:

(a) $p$ is a directed weak fibration,

(b) 1) For each directed space $\underline{X}$ and all d-maps $\Phi':X\times \uparrow [0,\varepsilon]\rightarrow \underline{E}$, $\varphi:\underline{X}\times  \uparrow \mathbf{I}$ with $p\circ \Phi'=\varphi|X\times \uparrow [0,\varepsilon]$, there exists a d-homotopy
$\Phi:\underline{X}\times \uparrow \mathbf{I}\rightarrow \underline{E}$, with $p\circ \Phi=\varphi$ and $\Phi_0=\Phi'_0$,

and

2)For each directed space $\underline{X}$ and all d-maps $\Phi':X\times \uparrow [\varepsilon,1]\rightarrow \underline{E}$, $\varphi:\underline{X}\times  \uparrow \mathbf{I}$ with $p\circ \Phi'=\varphi|X\times \uparrow [\varepsilon,1]$, there exists a d-homotopy
$\Phi:\underline{X}\times \uparrow \mathbf{I}\rightarrow \underline{E}$, with $p\circ \Phi=\varphi$ and $\Phi_1=\Phi'_1$.
\end{prop}
\begin{proof}
The implication (b)$\Rightarrow$ (a)is an immediate consequence of the obvious reciprocal result of Corollary \ref{cor.2.3.3}.

(a)$\Rightarrow$ (b)1). It is sufficient to consider only the case $\varepsilon=\frac{1}{2} $. Therefore we have given $\Phi':\underline{X}\times \uparrow [0,\frac{1}{2}]\rightarrow \underline{E}$, $\varphi:\underline{X}\times \uparrow \mathbf{I}\rightarrow \underline{B}$ with $p\circ \Phi'=
\varphi|\underline{X}\times \uparrow [0,\frac{1}{2}]$. Define $\widetilde{\varphi}:\underline{X}\times \uparrow [0,\frac{1}{2}]\times \uparrow
\mathbf{I}\rightarrow \underline{B}$ by $\widetilde{\varphi}(x,s,t):=\varphi(x,1-(1-s)(1-t))$. This is a d-map since the function $\uparrow
[0,\frac{1}{2}]\times \uparrow \mathbf{I}$, $(s,t)\rightarrow 1-(1-s)(1-t)$ is a d-map which satisfies $\widetilde{\varphi}(x,s,0)=p\Phi'(x,s)$.
Now by the dWCHP of $p$ with respect to the d-space $\underline{X}\times \uparrow [0,\frac{1}{2}]$ there exists $\widetilde{\Phi}:\underline{X}
\times \uparrow [0,\frac{1}{2}]\times \uparrow \mathbf{I}\rightarrow \underline{E}$ with $p\circ \widetilde{\Phi}=\widetilde{\varphi}$ and
$\widetilde{\Phi}_0\stackunder{(p)}{\simeq_d}\Phi'$. Suppose that $\Psi:\Phi'\stackunder{(p)}{\preceq_d}\widetilde{\Phi}_0$. Then define
$\Phi:\underline{X}\times \uparrow \mathbf{I}\rightarrow \underline{E}$ by
\[\Phi(x,t)=\left\{ \begin{array}{ll}
\Psi(x,t,2t)), & \mbox{if $0\leq t\leq 1/2 $},\\
\widetilde{\Phi}(x,\frac{1}{2},2t-1), & \mbox{if $1/2\leq t\leq 1 $}.\end{array} \right.
\]
This is well defined since $\Psi(x,\frac{1}{2},1)=\Psi_1(x,\frac{1}{2})=\widetilde{\Phi}_0(x,\frac{1}{2})=\widetilde{\Phi}(x,\frac{1}{2},0)$, and it
is obviously a d-map. Then, if $0\leq t\leq 1/2$, $p\Phi(x,t)=p\Psi(x,t,2t)=p\Phi'(x,t)=\varphi(x,t)$, and for $1/2\leq t\leq 1$, $p\Phi(x,t)=
p\widetilde{\Phi}(x,\frac{1}{2},2t-1)=\widetilde{\varphi}(x,\frac{1}{2},2t-1)=\varphi(x,t)$. Hence $p\circ \Phi=\varphi$. And $\Phi_0=\Psi_0=
\Phi'_0$.

If we have $\overline{\Psi}:\widetilde{\Phi}_0\stackunder{(p)}{\preceq_d}\Phi'$, define $\Phi:\underline{X}\times \uparrow \mathbf{I}\rightarrow \underline{E}$ by
\[\Phi(x,t)=\left\{ \begin{array}{ll}
\widetilde{\Phi}(x,\frac{1}{2},2t), & \mbox{if $0\leq t\leq 1/2 $},\\
\overline{\Psi}(x,t,2t-1), & \mbox{if $1/2\leq t\leq 1 $}.\end{array} \right.
\]
which verifies the same properties as above.

If $\Psi:\Phi'\stackunder{(p)}{\preceq_d}\chi$ and $\overline{\Psi}:\widetilde{\Phi}_0\stackunder{(p)}{\preceq_d}\chi$, define $\Phi$ by
\[\Phi(x,t)=\left\{ \begin{array}{ll}
\Psi(x,t,2t), & \mbox{if $0\leq t\leq 1/2 $},\\
\overline{\Psi}(x,t,2t-1), & \mbox{if $1/2\leq t\leq 1 $}.\end{array} \right.
\]
The other cases which imply the equivalence $\widetilde{\Phi}_0\stackunder{(p)}{\simeq_d}\Phi'$ are treated similarly.

(a)$\Rightarrow$ (b)2). We have given $\Phi':\underline{X}\times \uparrow [\frac{1}{2},1]\rightarrow \underline{E}$, $\varphi:\underline{X}\times \uparrow \mathbf{I}\rightarrow \underline{B}$ with $p\circ \Phi'=\varphi|\underline{X}\times \uparrow [\frac{1}{2},1]$. Define $\widetilde{\varphi}:\underline{X}\times \uparrow [\frac{1}{2},1]\times \uparrow\mathbf{I}\rightarrow \underline{B}$ by $\widetilde{\varphi}(x,s,t):=\varphi(x,st)$. This is a d-map and satisfies $\widetilde{\varphi}(x,s,1)=p\Phi'(x,s)$. Then continue similarly to
(b)1).
\end{proof}

\begin{lem}\label{lem.3.2.3}Let $p:E\rightarrow B$ be a directed weak fibration, $\underline{X}$ a d-space, $\uparrow A\subset  \uparrow V\subset \underline{X}$, $\uparrow V$ a d-halo around $\uparrow A$ in $\underline{X}$. Let $0<\varepsilon<1 $ and the commutative diagrams

a)
\[\xymatrix{\underline{E} \ar[r]^{p} & \underline{B}\\
\uparrow V\times \uparrow \mathbf{I} \cup (\underline{X}\times \uparrow[0,\varepsilon])\ar[u]^{\Phi'}\ar[r] & \underline{X}\times \uparrow \mathbf{I}\ar[u]_{\varphi}}\]

b)
\[\xymatrix{\underline{E} \ar[r]^{p} & \underline{B}\\
\uparrow V\times \uparrow \mathbf{I} \cup (\underline{X}\times \uparrow[\varepsilon,1])\ar[u]^{\Phi'}\ar[r] & \underline{X}\times \uparrow \mathbf{I}\ar[u]_{\varphi}}\]
where the lower horizontal arrows are inclusions.
Then there exists  $\Phi:\underline{X}\times \uparrow \mathbf{I}\rightarrow \underline{E}$ a d-homotopy lifting $\varphi$, $p\circ \Phi=
\varphi$ and such that:

a) $\Phi|(\uparrow A\times \uparrow \mathbf{I})\cup (\underline{X}\times \{0\})=\Phi'|(\uparrow A\times \uparrow \mathbf{I})\cup (\underline{X}\times \{0\})$,

b) $\Phi|(\uparrow A\times \uparrow \mathbf{I})\cup (\underline{X}\times \{1\})=\Phi'|(\uparrow A\times \uparrow \mathbf{I})\cup (\underline{X}\times \{1\})$.
\end{lem}
\begin{proof}Let $\tau_\alpha:\underline{X}\rightarrow \uparrow \mathbf{I},\alpha\in \{0,1\}$ be the haloing functions.

Case a). $A\subset \tau_1^{-1}(1)$, $CV\subset \tau_1^{-1}(0)$. Then define
$$\overline{\varphi}:\underline{X}\times \mathbf{I}\rightarrow \underline{B}, \qquad\overline{\varphi}(x,t):=\varphi(x,Min(\tau_1(x)+t,1)),$$
$$\overline{\Phi'}:X\times \uparrow [0,\varepsilon]\rightarrow \underline{E},\qquad \overline{\Phi'}(x,t):=\Phi'(x,Min(\tau_1(x)+t,1)).$$
The definition of $\overline{\Phi'}$ makes sense because $CV\subset \tau_1^{-1}(0)$. Also these maps are d-maps for reasons similar to those from the proof of Prop.\ref{prop.3.2.1}. Then $p\circ \overline{\Phi'}=\overline{\varphi}|\underline{X}\times \uparrow [0,\varepsilon]$. Now by  Prop.\ref{prop.3.2.2}, (a)$\Rightarrow$ (b)1), there exists a d-homotopy $\overline{\Phi}:\underline{X}\times \uparrow \mathbf{I }\rightarrow \underline{E}$ lifting $\overline{\varphi}$, $p\circ \overline{\Phi}=\overline{\varphi}$, with $\overline{\Phi}|\underline{X}\times \{0\}=\overline{\Phi'}|\underline{X}\times \{0\}$. Then we define $\Phi:\underline{X}\times \uparrow \mathbf{I}\rightarrow \underline{E}$ by

\[\Phi(x,t)=\left\{ \begin{array}{ll}
\Phi'(x,t), & \mbox{if $0\leq t\leq \tau_1(x) $},\\
\overline{\Phi}(x,t-\tau_1(x)), & \mbox{if $\tau_1(x)\leq t\leq 1 $}.\end{array} \right.
\]
This is well defined since for $t=\tau_1(x)$, $\overline{\Phi}(x,0)=\overline{\Phi'}(x,0)=\Phi'(x,\tau_1(x))$. It is also a d-map because the map
$\underline{X}\times \uparrow \mathbf{I}\rightarrow \underline{X}\times \uparrow \mathbf{I}$ , $(x,t)\rightarrow (x,\tau_1(x))$ is directed.
Then if $(x,t)\in \uparrow A\times\uparrow \mathbf{I}$ , then $\tau_1(x)=1$ and $\Phi(x,t)=\Phi'(x,t)$, $(\forall)t\in [0,1]$. For $(x,0)\in \underline{X}\times \{0\}$, $\Phi(x,0)=\Phi'(x,0)$. Finally, for $0\leq t\leq \tau_1(x)$, $p\Phi(x,t)=p\Phi'(x,t)=\varphi(x,t)$, and for $\tau_1(x)
\leq t\leq 1$, $p\Phi(x,t)=p\overline{\Phi}(x,t-\tau_1(x))=\overline{\varphi}(x,t-\tau_1(x))=\varphi(x,t)$. Therefore $p\circ \Phi=\varphi$.

Case b). $A\subset \tau_0^{-1}(0)$, $CV\subset \tau_1^{-1}(1)$ . Define
$$\widetilde{\varphi}:\underline{X}\times \uparrow \mathbf{I}\rightarrow \underline{B},\qquad \widetilde{\varphi}(x,t)=\varphi(x,Max(t+\tau_0(x)-1,0),$$
$$\widetilde{\Phi'}:\underline{X}\times \uparrow [\varepsilon,1]\rightarrow \underline{E},\qquad\widetilde{\Phi'}(x,t)=\Phi'(x,Max(t+\tau_0(x)-1,0).$$
These d-maps satisfy the relation $p\circ \widetilde{\Phi'}=\widetilde{\varphi}|\underline{X}\times \uparrow [\varepsilon,1]$.
Now by Prop.\ref{prop.3.2.2}, (a)$\Rightarrow$ (b)2), there exists a d-homotopy $\widetilde{\Phi}:\underline{X}\times \uparrow \mathbf{I}\rightarrow \underline{E}$ with $p\circ \widetilde{\Phi}=\widetilde{\varphi}$ and $\widetilde{\Phi}|\underline{X}\times\{1\}=\widetilde{\Phi'}|\underline{X}\times \{1\}$. Then we can define
\[\Phi(x,t)=\left\{ \begin{array}{ll}
\widetilde{\Phi}(x,t-\tau_0(x)+1), & \mbox{if $0\leq t\leq \tau_0(x) $},\\
\Phi'(x,t), & \mbox{if $\tau_0(x)\leq t\leq 1 $}.\end{array} \right.
\]
For $t=\tau_0(x)$, $\widetilde{\Phi}(x,t-\tau_0(x)+1)=\widetilde{\Phi}(x,1)=\widetilde{\Phi'}(x,1)=\Phi'(x,\tau_0(x))$. Hence $\Phi$ is well defined and it is also a d-homotopy. For $(x,t)\in \uparrow A\times\uparrow \mathbf{I}$ , we have $\tau_0(x)=0$ and $\Phi(x,t)=\Phi'(x,t)$, $(\forall)t\in [0,1]$. For $(x,1)\in \underline{X}\times \{1\}$, $\Phi(x,1)=\Phi'(x,1)$. Finally, for $\tau_0(x)\leq t\leq 1$, $p\Phi(x,t)=p\Phi'(x,t)=\varphi(x,t)$, and for $0\leq t\leq \tau_0(x)$, $p\Phi(x,t)=p\widetilde{\Phi}(x,t-\tau_0(x)+1)=\widetilde{\varphi}(x,t-\tau_0(x)+1)=\varphi(x, Max(t,0)=\varphi(x,t)$, i.e., $p\circ \Phi=\varphi$.
\end{proof}
\begin{thm}\label{thm.3.2.4}Let $p:\underline{E}\rightarrow \underline{B}$ be a directed map and $\{V(j)|j\in J\}$ a d-numerable covering of $\underline{B}$. If $p_{V(j)}:=p|p^{-1}(V(j)):\uparrow p^{-1}(V(j))\rightarrow \uparrow V(j)$ is a directed weak fibration for each $j\in J$, then $p$ is a directed weak fibration.
\end{thm}
\begin{proof}The proof follows the idea of the proof of Theorem 9.5 from \cite{Dieck}, but using the preparatory results (Prop.\ref{prop.3.2.1}, Prop.\ref{prop.3.2.2} and Lemma \ref{lem.3.2.3}) proved above in the directed topology context, as in this proof also.

Suppose given a commutative diagram
\[\xymatrix{\underline{E} \ar[r]^{p} & \underline{B}\\
\underline{X}\ar[u]^{f}\ar[r]_{\partial^\alpha} & \underline{X}\times \uparrow \mathbf{I}\ar[u]_{\varphi_\alpha}}\]
with $\alpha\in \{0,1\}$, and $\varphi_\alpha$ a semistationary d-homotopy, i.e., $\varphi_0(x,t)=\varphi_0(x,0), (\forall)t\in [0,\frac{1}{2}],
x\in \underline{X}$ and $\varphi_1(x,t)=\varphi_1(x,1),(\forall)t\in [\frac{1}{2},1], x\in \underline{X}$. We need to prove that there exists a
d-homotopy $\Phi_\alpha:\underline{X}\times \uparrow \mathbf{I}\rightarrow \underline{E}$ satisfying $p\circ \Phi_\alpha=\varphi_\alpha$ and
$\Phi_\alpha\circ \partial^\alpha=f$.

Consider the d-space
$$ \underline{W}_\alpha=\{(x,w)\in \underline{X}\times \underline{E}^{\uparrow \mathbf{I}}|f(x)=w(\alpha),pw(t)=\varphi(x,t)\},$$
and the d-map
$$ q_\alpha:\underline{W}_\alpha\rightarrow \underline{X}, \qquad q_\alpha(x,w)=x.$$
We can prove that the existence of the d-homotopy $\Phi_\alpha$ follows from the existence of a d-section for $q_\alpha$. Let $\uparrow A$ be a
d-subspace of $\underline{X}$. Denote by $S_\alpha(\uparrow A)$ the set of d-sections of $q_\alpha$ over $\uparrow A$, and by $F_\alpha(\uparrow A)$
the set of d-homotopies $\Phi_\alpha:\uparrow A\times \uparrow \mathbf{I}\rightarrow \underline{E}$, with $p\circ \Phi_\alpha=\varphi_\alpha|
\uparrow A\times \uparrow \mathbf{I}$ and $\Phi_\alpha(a,\alpha)=f(a)$, for $a\in \uparrow A$.

The mapping $F_\alpha(\uparrow A)\rightarrow S_\alpha(\uparrow A)$ given by the correspondence $\Phi^\alpha\rightarrow s_\alpha$, with $s_\alpha(a)=
(a,\Phi_\alpha^a)$, with $\Phi_\alpha^a(t)=\Phi_\alpha(a,t)$, is bijective, with the following inverse: if $s_\alpha:\uparrow A\rightarrow
\underline{W}_\alpha$ is a d-section for $q_\alpha$ over $\uparrow A$ given by $s_\alpha(a)=(a,w(a))$, then define $\Phi_\alpha(a,t)=w(a)(t)$, which
verifies $p\Phi_\alpha(a,t)=pw(a)(t)=\varphi_\alpha(a,t)$, and $\Phi_\alpha(a,\alpha)=w(a)(\alpha)=f(a)$.

Now to prove that a d-section for $q_\alpha$ exists we will apply Proposition \ref{prop.3.1.8}. At first, a d-numerable covering of $\underline{X}\times \uparrow \mathbf{I}$ is $\{U_\alpha(j):=\varphi_\alpha^{-1}(V(j))|j\in J\}$. For this by Proposition \ref{prop.3.2.1}
there exists a d-numerable covering $\{X_k|k\in K\}$ of $\underline{X}$ and a family of positive real numbers $\{\varepsilon_k|k\in K$, such that
for $t_1,t_2\in [0,1]$ with $0\leq t_2-t_1<\varepsilon_k$, there is a $j\in J$ with $X_k\times [t_1,t_2]\subset U_\alpha(j)$. Now by Prop.\ref{prop.3.1.8} it is sufficient to prove that the maps $q_{\alpha k}:q_\alpha|q_\alpha^{-1}(\uparrow X_k):\uparrow q_\alpha^{-1}(X_k)\rightarrow
\uparrow X_k$ have the d-SEP. More generally, we prove that for every subset $Z\subset X_k$ the restrictions $q_{\alpha Z}:=q_\alpha|\uparrow q_\alpha^{-1}(Z)$ have the d-SEP.

Let $\uparrow A $ be a d-subspace of $\uparrow Z$ and $\uparrow V$ a d-halo of $\uparrow A$ in $\uparrow Z$. Suppose that $s_{\uparrow V}$ is a d-section of $q_{\alpha Z}$ over $\uparrow V$. This define a d-homotopy $\Phi_{\alpha V}:\uparrow V\times \uparrow \mathbf{I}\rightarrow \underline{E}$
with $p\circ \Phi_{\alpha V}=\varphi_\alpha|\uparrow V\times \uparrow \mathbf{I}$ and $\Phi_{\alpha V}(z,\alpha)=f(z),(\forall) z\in \uparrow Z$.

Now the cases $\alpha=0$ and $\alpha=1$ need to be kept separate.

Case $\alpha=0$. Consider $0\leq t_1<\frac{1}{2}$. For this we will define a d-homotopy
$$\Phi^0_0:\uparrow V\times \uparrow \mathbf{I}\bigcup \uparrow Z\times \uparrow[0,t_1]\rightarrow \underline{E}$$
over $\varphi_0$, and satisfying $ \Phi^0_0|\uparrow A\times \uparrow \mathbf{I}=\Phi_{0V}|\uparrow A\uparrow\times \uparrow \mathbf{I}$ and
$\Phi^0_0(z,0)=f(z)$, for $z\in V$. For this purpose we consider $\tau_1:\uparrow Z\rightarrow \uparrow \mathbf{I}$ a d-haloing function with
$\uparrow A\subset \tau_1^{-1}(1)$ and $Z\backslash V\subset \tau_1^{-1}(0)$. For $z\in \uparrow Z$ with $\tau_1(z)\leq t_1$ define
$\theta_z:\uparrow \mathbf{I}\rightarrow \uparrow \mathbf{I}$ an increasing piecewise affine function with $\theta_z(0)=0,\theta_z(\tau_1(z))=
\tau_1(z), \theta_z(t_1)=\tau_1(z),\theta_z(\frac{1}{2}=\frac{1}{2},\theta_z(1)=1$. Then we define $\Phi^0_0$ by

\[\Phi^0_0(z,t)=\left\{ \begin{array}{ll}
\Phi_{0V}(z,\theta_z(t)), & \mbox{if $z\in V,t\in [0,1] $},\\
f(z), & \mbox{if $z\in \tau_1^{-1}(0),0\leq t\leq t_1 $}.\end{array} \right.
\]
This is well defined since if $z\in \tau_1^{-1}(0)$ then $z\overline{\in}V$. It is also a d-map since $\Phi_{0V}$ and $\theta_z$ are d-maps.
And $\Phi^0_0$ satisfies $\Phi^0_0(z,0)=f(z)$ obviously and $\Phi^0_0(a,t)=\Phi_{0V}(a,t)$ , for $a\in \uparrow A$ since $ \tau_1(a)=1$ and so
$\theta_a=id_{\uparrow \mathbf{I}}$.

Now choose $0=t_0<t_1<...<t_n=1$ such that $t_1<\frac{1}{2}$ and $t_{i+1}-t_i<\varepsilon_k, i=1,...,n-1$. For this consider
$\uparrow W_0^i:=\uparrow \tau_1^{-1}[t_i,1]$, for $i=1,2,...,n$ and $\uparrow W_0^0=\uparrow V$. Then for $1\leq i\leq n-1$ we construct
inductively $\Phi_0^i:\uparrow Z\times \uparrow [t_{i-1},t_{i+1}]\rightarrow \underline{E}$ over $\varphi_0$ by using Lemma \ref{lem.3.2.3}a),
with $\uparrow W_0^i$ as $\uparrow A$, $\uparrow W_0^{i-1}$ as $\uparrow V$ , $\uparrow Z$ as $ \underline{X} $ , $\uparrow [t_{i-1},t_{i+1}]$ as
$\uparrow \mathbf{I}$ and $\uparrow [t_{i-1},t_i]$ as $\uparrow [0,\varepsilon]$. Now we define $\Phi_{\uparrow Z}:\uparrow Z\times \uparrow \mathbf{I}\rightarrow \underline{E}$ by $\Phi_{\uparrow Z}(z,t)=\Phi_0^i(z,t)$ for $t\in [t_{i-1},t_{i+1}]$, $i<n-1$, and $\Phi_Z(z,t)=
\Phi A_0^{n-1}$ for $t\in [t_{n-2},t_n]$. This satisfies the conditions $p\circ \Phi_{\uparrow Z}=\varphi_0|\uparrow Z\times \uparrow \mathbf{I}$ and
$\Phi_{\uparrow Z}\circ \partial^0=f|\uparrow Z$. Therefore $\Phi_{\uparrow Z}$ defines a d-section of $q_{0Z}$ which extends $s_{\uparrow V}|
\uparrow A$ because $\Phi_{\uparrow Z}|\uparrow A\times \uparrow \mathbf{I}=\Phi_{0V}|\uparrow A\times \uparrow \mathbf{I}$.

Case $\alpha=1$. Repeat the construction from the case $\alpha=0$ by a 'symmetry'. Consider $\frac{1}{2}<t_{n-1}\leq 1$ for which we define a
d-homotopy
$$\Phi^1_1:\uparrow V\times \uparrow \mathbf{I}\bigcup \uparrow Z\times \uparrow [t_{n-1},1]\rightarrow E,$$
by

\[\Phi^1_1(z,t)=\left\{ \begin{array}{ll}
f(z), & \mbox{if $z\in \tau_0^{-1}(0),t_{n-1}\leq t\leq\leq 1 $},\\
\Phi_{1V}(z,\theta'_z(t)), & \mbox{if $z\in V,t\in [0,1] $},\end{array} \right.
\]
where $\tau_0:\uparrow Z\rightarrow \uparrow \mathbf{I}$ satisfies $\uparrow A\subset \tau_0^{-1}(0)$ and $Z\backslash V\subset \tau_0^{-1}(1)$,
and $\theta'_z$ is defined for $\tau_0(z)\geq t_{n-1}$ as an increasing piecewise affine function $\theta'_z:\uparrow \mathbf{I}\rightarrow \uparrow \mathbf{I}$ with $\theta'_z(0)=0,\theta'_z(\frac{1}{2})=\frac{1}{2},\theta'_z(t_{n-1})=\tau_0(z),\theta'_z(\tau_0(z))=\tau_0(z),\tau'_z(1)=1$.
Then $\Phi^1_1|\uparrow A\times \uparrow \mathbf{I}=\Phi_{1V}|\uparrow A\times \uparrow \mathbf{I}$.

Then consider a division of $[0,1]$ by $0=t_0<t_1<...<t_{n-1}<t_n=1$ such that $t_{n-1}>\frac{1}{2}$ and $t_{i+1}-t_i<\varepsilon_k$, for $i=0,1,...,
n-2$. For this we denote $\uparrow W^i_1=\tau_0[0,t_i]$, $i=0,1,...,n-1$, and $\uparrow W^n_1=\uparrow V$. Then for $0\leq i\leq n-2$, we construct $\Phi^i_1:\uparrow Z\times \uparrow [t_{i-1},t_{i+1}]\rightarrow \underline{E}$, over $\varphi_1$, by using Lemma \ref{lem.3.2.3} b), with $\uparrow
W^i_1$ as $\uparrow A$, $\uparrow W_1^{i+1}$ as $\uparrow V$, $\uparrow Z$ as $\underline{X}$, $\uparrow [t_{i-1},t_{i+1}]$ as $\uparrow \mathbf{I}$
and $\uparrow [t_{i-1},t_i]$ as $[\varepsilon,1]$. Now we can define $\Phi'_{\uparrow Z}:\uparrow Z\times \uparrow \mathbf{I}\rightarrow E$, by
$\Phi'_{\uparrow Z}(z,t)=\Phi^i_1(z,t)$ for $t\in [t_{i-1},t_{i+1}],i>1$, and $\Phi'_{\uparrow Z}(z,t)=\Phi^1_1(z,t)$ for $t\in [0,t_2]$. This
d-homotopy satisfies the conditions $p\circ \Phi'_{\uparrow Z}=\varphi_1|\uparrow Z\times \uparrow \mathbf{I}$ and $\Phi'_{\uparrow Z}\circ
\partial^1=f|\uparrow Z$. Therefore $\Phi'_{\uparrow Z}$ defines a d-section for $q_{1Z}$ which extend $s_{\uparrow V}|\uparrow A$ because
$\Phi'_{\uparrow Z}|\uparrow A\times \uparrow \mathbf{I}=\Phi_{1V}|\uparrow A\times \uparrow \mathbf{I}$. With this and the explanations from the
beginning of these constructions the proof is complete.
\end{proof}
\begin{cor}\label{cor.3.2.5}Let $p:\underline{E}\rightarrow \underline{B}$ be a directed map and $\{V(j)|j\in J\}$ a d-numerable covering of $\underline{B}$. If every $p_{V(j)}:=p/p^{-1}(V(j)):\uparrow p^{-1}(V(j))\rightarrow \uparrow V(j)$ is directed dominated by a directed fibration (e.g., by a trivial d-space over each $\uparrow V(j)$),then $p$ is a directed weak fibration.
\end{cor}

Faculty of Mathematics

"Al. I. Cuza" University,

 700505-Iasi, Romania.

E-mail:ioanpop@uaic.ro

\end{document}